\documentclass[12pt,reqno]{amsart}   	% use "amsart" instead of "article" for AMSLaTeX format
\usepackage{bm} 
\usepackage{amsfonts}
\usepackage{mathrsfs}
\usepackage{yfonts}
\usepackage{url}

\usepackage[osf]{mathpazo}  % This provides the palatino font and some nice "old style" numbers.

\usepackage{mathtools}
\usepackage{amssymb}  		% Activate to begin paragraphs with an empty line rather than an indent
\usepackage{graphicx}
\usepackage{verbatim}
\usepackage{enumerate}	
\usepackage[english]{babel}
\usepackage{amsmath}
\usepackage{bussproofs}	
\usepackage{quoting}
\usepackage{array,booktabs}

\usepackage[shortlabels]{enumitem}
\usepackage{units}
\usepackage{xspace}\xspaceaddexceptions{\,}
\usepackage{multicol}
\usepackage{booktabs,caption}
   \usepackage{tikz-cd}
\quotingsetup{font=small}	% Use pdf, png, jpg, or eps§ with pdflatex; use eps in DVI mode
\usepackage{amsthm}								% TeX will automatically convert eps --> pdf in pdflatex	
\usepackage%[final]
[authormarkup=none]
{changes}

\definechangesauthor[color=red]{Stef}

\newcommand*\circled[1]{\tikz[baseline=(char.base)]{
            \node[shape=circle,draw,inner sep=4pt] (char) {#1};}}

\usepackage[pdftex,bookmarks,bookmarksnumbered,linktocpage,   %%  customize to your liking
         colorlinks,linkcolor=blue,citecolor=blue]{hyperref}

\theoremstyle{definition}
\newtheorem{definition}{Definition}
\newtheorem{example}[definition]{Example}

\newtheorem{remark}[definition]{Remark}

\theoremstyle{plain}
\newtheorem{lemma}[definition]{Lemma}

\newtheorem{theorem}[definition]{Theorem}
\newtheorem{corollary}[definition]{Corollary}

\newcommand\A{{\mathbf A}}
\newcommand\B{{\mathbf B}}
\newcommand\C{{\mathbf C}}

\newcommand\Fm{\mathbf{Fm}}

\newcommand\CL{\ensuremath{\mathrm{CL}}\xspace}

\newcommand\PWK{\ensuremath{\mathrm{PWK}}\xspace}

\bmdefine{\boldstar}{\mathchoice{\textstyle*}{\textstyle*}{\textstyle*}{\scriptstyle*}}
\newcommand\Modstar{\mathsf{Mod}^{\boldstar}}

\newcommand{\ModS}{\mathsf{Mod}^{\textup{Su}}}

\newcommand\Alg[1]{\if#1*\operatorname{\mathsf{Alg}*}\else\operatorname{\mathsf{Alg}}#1\fi}

\newcommand\Mod[1]{\if#1*\operatorname{\mathsf{Mod}*}\else\operatorname{\mathsf{Mod}}#1\fi}

\newcommand{\VVV}{\mathbb{V}}                     %  class operators

\newcommand{\PPP}{\mathbb{P}}
\newcommand{\PSD}{\mathbb{P}_{\!\textsc{sd}}^{}}

\newcommand{\SSS}{\mathbb{S}}
\newcommand{\III}{\mathbb{I}}

\bmdefine{\Leibniz}{\Omega}        %  Leibniz operator

\bmdefine{\frege}{\Lambda}         %  Frege operator

\makeatletter
\newcommand{\tarskidsp}{\mathord%
   {\m@th\raisebox{0pt}[0pt][0pt]{$\stackrel%
   {\raisebox{-2.7pt}[0ex][0pt]{$\displaystyle \,\?\thicksim$}}%
   {\displaystyle\Leibniz}$}}}
\newcommand{\tarskitxt}{\mathord%
   {\m@th\raisebox{0pt}[0pt][0pt]{$\stackrel%
   {\raisebox{-2.7pt}[0ex][0pt]{$\,\?\thicksim$}}{\displaystyle\Leibniz}$}}}
\newcommand{\tarskiscr}{\mathord%
   {{\m@th\raisebox{0pt}[0pt][0pt]{$\stackrel%
   {\raisebox{-2.4pt}[0ex][0pt]{$\scriptstyle \,\?\thicksim$}}%
   {\scriptstyle\Leibniz}$}}}}
\newcommand{\tarskiscrscr}{\mathord%
   {{\m@th\raisebox{0pt}[0pt][0pt]{$\stackrel%
   {\raisebox{-2pt}[0ex][0pt]{$\scriptscriptstyle \,\?\thicksim$}}%
   {\scriptscriptstyle\Leibniz}$}}}}
\newcommand{\Tarski}{\@ifnextchar ^ %
   {\mathchoice{\tarskidsp\kern-.07em}{\tarskitxt\kern-.07em}%
   {\tarskiscr\kern-.07em}{\tarskiscrscr\kern-.07em}}%
   {\mathchoice{\tarskidsp}{\tarskitxt}{\tarskiscr}{\tarskiscrscr}}}
\makeatother
\DeclareMathAlphabet{\mathbfsf}{\encodingdefault}{\sfdefault}{bx}n
\makeatletter
\providecommand*{\Dashv}{\mathrel{\mathpalette\@Dashv\vDash}}
\newcommand*{\@Dashv}[2]{\reflectbox{$\m@th#1#2$}}
\makeatother

\renewcommand\geq{\geqslant}

\newcommand\pair[1]{{\langle#1\rangle}}

\useshorthands{"}
\defineshorthand{"-}{{}\nobreakdash-\hspace{0pt}}	% nonbreak dash

\newcommand{\FFi}{\mathcal{F}i}

\newcommand\PL{{\mathcal{P}}_{\textit{\l}}}

\EnableBpAbbreviations

\newcommand{\bit}{\begin{itemize}}    % but see also \benbullet below
\newcommand{\eit}{\end{itemize}}
\newcommand{\ben}{\begin{enumerate}}
\newcommand{\een}{\end{enumerate}}
\newcommand{\benroman}{\ben[\normalfont (i)]}  % *
\let\eroman\een
\newcommand{\bde}{\begin{description}}
\newcommand{\ede}{\end{description}}

\newcommand{\Var}{\mathnormal{V\mkern-.8\thinmuskip ar}} % sentential variables
\newcommand{\?}{\ensuremath{\mkern0.4\thinmuskip}}   % very small math space
\newcommand{\sineq}{\mathrel{\dashv\mkern1.5mu\vdash}}  %  interderivability

\begin{document}

\title{Logics of left variable inclusion and P\l onka sums of matrices}
\subjclass[2010]{Primary: 03G27. Secondary: 03G25.}

\keywords{P\l onka sums, Kleene logics, abstract algebraic logic, regular varieties}
\date{\today}

\author{S. Bonzio}
\address{Stefano Bonzio \\
Department of Biomedical Sciences and Public Health, 
Polytechnic University of the Marche, Ancona, Italy.}
\email{stefano.bonzio@gmail.com}
\author{T. Moraschini}
\address{Tommaso Moraschini\\
Department of Philosophy, Faculty of Philosophy, Carrer Montalegre 6, Barcelona.}
\email{tommaso.moraschini@gmail.com}
\author{M. Pra Baldi}
\address{Michele Pra Baldi \\
Department of Pedagogy, Psychology and Philosophy,
University of Cagliari, Italy.}
\email{m.prabaldi@gmail.com}
%\keywords{}

%\institute{}

\maketitle

\begin{abstract}
The paper aims at studying, in full generality, logics defined by imposing a variable inclusion condition on a given logic $\vdash$. We prove that the description of the algebraic counterpart of the left variable inclusion companion of a given logic $\vdash$ is related to the construction of P\l onka sums of the matrix models of $\vdash$. This observation allows to obtain a Hilbert-style axiomatization of the logics of left variable inclusion, to describe the structure of their reduced models, and to locate them in the Leibniz hierarchy. 
\end{abstract}

\section{introduction}

It is always possible to associate with an arbitrary propositional logic $\vdash$, two new substitution-invariant consequence relations $\vdash^{l}$ and $\vdash^{r}$, which satisfy respectively a \textit{left} and a \textit{right variable inclusion principle}, as follows:
\[
\Gamma \vdash^{l} \varphi \Longleftrightarrow \text{ there is }\Delta \subseteq \Gamma \text{ s.t. }\Var(\Delta)\subseteq\Var(\varphi) \text{ and } \Delta\vdash\varphi,
\]
and
\[
\Gamma\vdash^{r}\varphi\iff \left\{ \begin{array}{ll}
\Gamma\vdash\varphi  \ \text{and} \ \Var(\varphi)\subseteq\Var(\Gamma),  \text{  or}\\
\Sigma\subseteq\Gamma, \text{  with $\Sigma$ a set of inconsistency terms for $\vdash$.} \\
  \end{array} \right.  
  %\Gamma \vdash^{r} \varphi \Longleftrightarrow \Var(\varphi)\subseteq\Var(\Gamma).
\]

Accordingly, we say that the logics $\vdash^{l}$ and $\vdash^{r}$ are, respectively, the \textit{left} and the \textit{right variable inclusion companions} of $\vdash$, sometimes also referred to as \emph{contaiment logics}.

Prototypical examples of variable inclusion companions are found in the realm of three-valued logics. For instance, the left and the right variable inclusion companions of classical (propositional) logic are respectively \textit{paraconsistent weak Kleene logic} (\PWK for short) \cite{Hallden, Kleene}, and \textit{Bochvar logic} \cite{Bochvar}. The fact that these logics coincide with the variable inclusion companions of classical logic was shown in \cite{CiuniCarrara,Urquhart2001}. Remarkably, both \PWK and Bochvar logic feature the presence of a non-sensical, infectious truth value \cite{Szmuc,Ciuni2}, which made them a valuable tool in modeling reasonings with non-existing objects \cite{Prior}, computer-programs affected by errors \cite{Ferguson} as well as recent developments in the theory of truth \cite{Szmuctruth} and philosophy of logic \cite{BoemBonzio}.

Recent work \cite{Bonzio16} linked \PWK to the algebraic theory of regular varieties, i.e.\ equational classes axiomatized by equations $\varphi \thickapprox \psi$ such that $\Var(\varphi) = \Var(\psi)$. The representation theory of regular varieties is largely due to the pioneering work of P\l onka \cite{Plo67}, and is tightly related to a special class-operator $\PL(\cdot)$ nowadays called \textit{P\l onka sums}. Over the years, regular varieties have been studied in depth both from a purely algebraic perspective \cite{Plo67a,Kal71,Harding2016, Harding20172} and in connection to their topological duals \cite{GR91, Loi, Romanowska97,SB18,Ledda2018}. The machinery of P\l onka sums has also found useful applications in the study of the constraint satisfaction problem \cite{Bergman2015} and database semantics \cite{Libkin,Puhlmann} and in the application of algebraic methods in computer science \cite{BonzioValota}.

One of the main results of \cite{Bonzio16} states that the algebraic counterpart of \PWK is the class of P\l onka sums of Boolean algebras. This observation led us to investigate the relations between left variable inclusion companions and P\l onka sums in full generality.\footnote{A similar investigation of right variable inclusion companions is developed in \cite{BonzioPraBa}.} Our study is carried on in the conceptual framework of abstract algebraic logic \cite{Cz01,Font16,FJa09}.

We begin by generalizing the construction of P\l onka sums from algebras to logical matrices (Section \ref{sec: definizione e primi risultati}). This allows us to condense the connection between left variable inclusion principles and P\l onka sums in the following slogan: The left variable inclusion companion $\vdash^{l}$ of a logic $\vdash$ is complete w.r.t.\ the class of P\l onka sums of matrix models of $\vdash$ (Corollary \ref{cor:completeness}).

As a matter of fact, left variable inclusion companions $\vdash^{l}$ are especially well-behaved in case the original logic $\vdash$ has a \textit{partition function} \cite{romanowska2002modes}, a feature shared by the vast majority of non-pathological logics in the literature. The importance of partition functions is reflected both at a syntactic and at a semantic level. Accordingly, on the one hand we present a general method to transform every Hilbert-style calculus for a finitary logic $\vdash$ with a partition function into a Hilbert-style calculus for $\vdash^{l}$ (Theorem \ref{theor:HilbertCalculus}). On the other hand, partition functions can be exploited to tame the structure of the matrix semantics $\ModS(\vdash^{l})$ of $\vdash^{l}$, given by the so-called Suszko reduced models of $\vdash^{l}$. In particular, we obtain  a full description of $\ModS(\vdash^{l})$ in case $\vdash$ is a finitary equivalential logic with a partition function (Theorems \ref{Thm: caratterizzazione suszko regolare} and \ref{Thm: caratterizzazione suszko regolare2}). We close our investigation by determining the location of $\vdash^{l}$ in the Leibniz hierarchy (Section  \ref{sec: Leibniz hierarchy}).

%\
%\
%\
% the construction of P\l onka sum from algebraic structures to logical matrices. The first achievement of the present paper consist in showing that (some) logics of variable inclusion, more precisely the left variable inclusion companions of a logic $\vdash$, admits a purely semantical description in terms of P\l onka sums of matrix models of the logic $\vdash$. 
%
%\
%\

\section{Preliminaries}\label{sec: preliminari}

\subsection*{Abstract Algebraic Logic}

For standard background on universal algebras and abstract algebraic logic we refer the reader respectively to  \cite{Be11g,BuSa00,DeWi02,McMcTa87} and \cite{BP89,BP86,BP92,Cz01,Font16,FJa09,FJaP03b,W88}.  In this paper, algebraic languages are assumed not to contain constant symbols.\ Moreover, unless stated otherwise, we work within a fixed but arbitrary algebraic language. We denote algebras by $\A, \B, \C\dots$ respectively with universes $A, B, C \dots$ A class of algebras is a \textit{variety} if is axiomatized by equations. Given a class of algebras $\mathsf{K}$, we denote by $\VVV(\mathsf{K})$ the variety generated by $\mathsf{K}$. 
Let $\Fm$ be the algebra of formulas built up over a countably infinite set $\Var$ of variables. Given a formula $\varphi\in Fm$, we denote by $\Var(\varphi)$ the set of variables really occurring in $\varphi$. Similarly, given $\Gamma\subseteq Fm$, we set
\[
\Var(\Gamma)=\bigcup \{\Var(\gamma)\colon \gamma\in\Gamma\}.
\]
A \emph{logic} is a substitution invariant consequence relation $\vdash \?\? \subseteq \mathcal{P}(Fm) \times Fm$, namely for every substitution $\sigma \colon \Fm \to \Fm$,
\[
\text{if }\Gamma \vdash \varphi \text{, then }\sigma [\Gamma] \vdash \sigma (\varphi).
\]
Given $\varphi, \psi \in Fm$, we write $\varphi \sineq \psi$ as a shorthand for $\varphi \vdash \psi$ and $\psi \vdash \varphi$. Moreover, we denote by $\mathrm{Cn}_{\vdash} \colon \mathcal{P}(Fm) \to \mathcal{P}(Fm)$ the closure operator associated with $\vdash$. A logic $\vdash$ is \emph{finitary} when the following holds for all $\Gamma\cup\{\varphi\}\subseteq Fm$:
\begin{align*}
\Gamma\vdash\varphi \Longleftrightarrow \exists \Delta \subseteq\Gamma \text{ s.t. } \Delta \text{ is finite and } \Delta\vdash\varphi.
\end{align*}

 A \emph{matrix} is a pair $\langle \A, F\rangle$ where $\A$ is an algebra and $F \subseteq A$. In this case, $\A$ is called the \textit{algebraic reduct} of the matrix $\langle \A, F \rangle$.  We denote by $\III, \SSS, \PPP$ and $\PSD$ respectively the class operators of isomorphic copies, substructures, direct products and subdirect products, which apply both to classes of algebras and classes of matrices.
 
 Every class of matrices $\mathsf{M}$ induces a logic as follows:
\begin{align*}
\Gamma \vdash_{\mathsf{M}} \varphi \Longleftrightarrow& \text{ for every }\langle \A, F \rangle \in \mathsf{M} \text{ and homomorphism }h \colon \Fm \to \A,\\
& \text{ if }h[\Gamma] \subseteq F\text{, then }h(\varphi) \in F.
\end{align*}
A logic $\vdash$ is \emph{complete} w.r.t.\ a class of matrices $\mathsf{M}$ when it coincides with $\vdash_{\mathsf{M}}$. 

 A matrix $\langle \A, F\rangle$ is a \emph{model} of a logic $\vdash$ when
\begin{align*}
\text{if }\Gamma \vdash \varphi, &\text{ then for every homomorphism }h \colon \Fm \to \A,\\  
&\text{ if }h[\Gamma] \subseteq F\text{, then }h(\varphi) \in F.
\end{align*}
A set $F \subseteq A$ is a (deductive) \textit{filter} of $\vdash$ on $\A$, or simply a $\vdash$-\textit{filter}, when the matrix $\langle \A, F \rangle$ is a model of $\vdash$. We denote by $\FFi_{\vdash}\A$ the set of all filters of $\vdash$ on $\A$, which turns out to be a closure system. Moreover, we denote by $\textup{Fg}_{\vdash}^{\A}(\cdot)$ the closure operator of $\vdash$-filter generation on $\A$.

Let $\A$ be an algebra and $F \subseteq A$. A congruence $\theta$ of $\A$ is \emph{compatible} with $F$ when for every $a,b\in A$,
\[
\text{if }a\in F\text{ and }\langle a, b \rangle \in \theta\text{, then }b\in F.
\]
It turns out that there exists the largest congruence of $\A$ which is compatible with $F$. This congruence is called the \emph{Leibniz congruence} of $F$ on $\A$, and it is denoted by $\Leibniz^{\A}F$. 

Let $\A$ be an algebra, $F \subseteq A$ and $\vdash$ be a logic. The \emph{Suszko congruence} of $F$ on $\A$, is defined as
\[
\Tarski^{\A}_{\vdash}F \coloneqq \bigcap \{ \Leibniz^{\A}G : F \subseteq G \text{ and }G \in \FFi_{\vdash}\A \}.
\]

 Let $\A$ be an algebra. A function $p \colon A^{n} \to A$ is a \textit{polynomial function} of $\A$ if there are a natural number $m$, a formula $\varphi(x_{1}, \dots, x_{n+m})$, and elements $b_{1}, \dots, b_{m} \in A$ such that
\begin{align*}
p(a_{1}, \dots, a_{n}) = \varphi^{\A}(a_{1}, \dots, a_{n}, b_{1}, \dots, b_{m})
\end{align*}
for every $a_{1}, \dots, a_{n} \in A$. 
\begin{lemma}\cite[Thm.\ 4.23]{Font16}\label{lem: polynomial charact of leib cong}
Let $\A$ be an algebra, $F \subseteq A$ and $a, b \in A$. 
\begin{align*}
\langle a, b \rangle \in \Leibniz^{\A}F \Longleftrightarrow & \text{ for every unary pol. function }p \colon A \to A,\\
&\text{ } p(a) \in F \text{ if and only if }p(b) \in F.
\end{align*}
\end{lemma}

\begin{lemma}\cite[Thm.\ 5.32]{Font16}\label{lem:polynomial charact. Suszko cong.}
Let $\vdash$ be a logic, $\A$ be an algebra, $F \subseteq A$ and $a, b \in A$. 
\begin{align*}
\langle a, b \rangle \in \Tarski^{\A}_{\vdash}F \Longleftrightarrow & \text{ for every unary pol. function }p \colon A\to A,\\
&\text{ } \textup{Fg}_{\vdash}^{\A}(F\cup\{p(a)\}) =  \textup{Fg}_{\vdash}^{\A}(F\cup\{p(b)\}).
\end{align*}
\end{lemma}

The Leibniz and Suszko congruences allow to associate two distinguished classes of models to logics. More precisely, given a logic $\vdash$, we set
\begin{align*}
\Mod(\vdash) & \coloneqq \{ \langle \A, F \rangle : \langle \A, F \rangle \text{ is a model of }\vdash \};\\
\Modstar(\vdash) & \coloneqq \{ \langle \A, F \rangle \in \Mod(\vdash) : \Leibniz^{\A}F \text{ is the identity} \};\\
\ModS(\vdash) & \coloneqq \{ \langle \A, F \rangle \in \Mod(\vdash) : \Tarski_{\vdash}^{\A}F \text{ is the identity} \}.
\end{align*}
The above classes of matrices are called, respectively, the classes of \text{models}, \textit{Leibniz reduced models}, and \textit{Suszko reduced models} of $\vdash$. It turns out that $\ModS(\vdash) = \PSD \Modstar(\vdash)$.

Trivial matrices will play a useful role in the whole paper. More precisely, a matrix $\langle \A, F \rangle$ is \textit{trivial} if $F = A$. We denote by $\langle \boldsymbol{1}, \{ 1 \} \rangle$ the trivial matrix, where $\boldsymbol{1}$ is the trivial algebra. Observe that the latter matrix is a model (resp.\ Leibniz and Suszko reduced model) of every logic. Moreover, if $\vdash$ is a logic and $\langle \A, F\rangle \in \ModS(\vdash)$ is a trivial matrix, then $\langle \A, F\rangle = \langle \boldsymbol{1}, \{ 1 \}\rangle$.

Given a logic $\vdash$, we set
\[
\Alg(\vdash) = \{ \A: \text{there is }F \subseteq A \text{ s.t. }\langle \A, F \rangle \in \ModS(\vdash) \}.
\]
In other words, $\Alg(\vdash)$ is the class of algebraic reducts of matrices in $\ModS(\vdash)$. The class $\Alg(\vdash)$ is called the \textit{algebraic counterpart} of $\vdash$. For the vast majority of logics $\vdash$, the class $\Alg(\vdash)$  is the class of algebras intuitively associated with $\vdash$.

\begin{lemma}\cite[Lemma 5.78]{Font16}\label{lem:algebraicreducts}
Let $\vdash$ be a logic defined by a class of matrices $\mathsf{M}$. Then $\Alg(\vdash) \subseteq \mathbb{V}(\mathsf{K})$, where $\mathsf{K}$ is the class of algebraic reducts of $\mathsf{M}$.
\end{lemma}

\begin{lemma}\label{lemma su equazioni p-function}
Let $\vdash$ be a logic and $\epsilon,\delta\in Fm $. The following are equivalent:
\benroman
\item$\Alg(\vdash)\vDash\epsilon\approx\delta$;
\item $\varphi(\epsilon,\vec{z}\?\?)\sineq\varphi(\delta,\vec{z}\?\?)$, for every formula $\varphi(v,\vec{z}\?\?)$.
\eroman
\end{lemma}
\begin{proof}
See \cite[Lemma 5.74(1)]{Font16} and \cite[Theorem 5.76]{Font16}.
\end{proof}

Now, we turn out attention to a fundamental topic in abstract algebraic logic, that is the so-called \textit{Leibniz hierarchy}, see for example \cite{Font16,JGRa11, JanMorI, JanMorII, JanMorIII}.\ We review only the material which is necessary for the present purpose. A logic $\vdash$ is \textit{protoalgebraic} if there is a set of formulas $\Delta(x, y)$ such that
\[
\emptyset \vdash \Delta(x, x) \text{ and }x, \Delta(x, y) \vdash y.
\]
Remarkably, $\vdash$ is protoalgebraic if and only if $\Modstar(\vdash) = \ModS(\vdash)$.

A logic $\vdash$ is \textit{equivalential} if there is a set of formulas $\Delta(x, y)$ such that for every $\langle \A, F \rangle \in \Mod(\vdash)$,
\[
\langle a, b \rangle \in \Leibniz^{\A}F \Longleftrightarrow \Delta^{\A}(a, b) \subseteq F\text{, for all }a, b \in A.
\]
In this case, $\Delta(x, y)$ is said to be a set of \textit{congruence formulas} for $\vdash$. Remarkably, $\vdash$ is equivalential if and only if $\Modstar(\vdash)$ is closed under $\SSS$ and $\PPP$. Consequently, every equivalential logic is protoalgebraic.

A logic $\vdash$ is \textit{truth-equational} if there is a set of equations $\boldsymbol{\tau}(x)$ such that for all $\langle \A, F \rangle \in \Modstar(\vdash)$,
\[
a \in F \Longleftrightarrow \A \vDash \boldsymbol{\tau}(a)\text{, for all }a\in A.
\]
In this case, $\boldsymbol{\tau}(x)$ is said to be a set of \textit{defining equations} for $\vdash$.

Finally, a logic $\vdash$ is \textit{algebraizable} when it is both equivalential and truth-equational. In this case, $\Alg(\vdash)$ is called the \textit{equivalent algebraic semantics} of $\vdash$.

\subsection*{P\l onka sums}

For standard information on P\l onka sums we refer the reader to \cite{Plo67a, Plo67,Romanowska92, romanowska2002modes}. A \textit{semilattice} is an algebra $\A = \langle A, \lor\rangle$, where $\lor$ is a binary commutative, associative and idempotent operation. Given a semilattice $\A$ and $a, b \in A$, we set
\[
a \leq b \Longleftrightarrow a \lor b = b.
\]
It is easy to see that $\leq$ is a partial order on $A$.
\begin{definition}\label{Def:Directed-System-Matrices}
A \textit{directed system of algebras} consists of:
\benroman
\item a semilattice $I = \langle I, \lor\rangle$;
\item a family of algebras $\{ \A_{i} : i \in I \}$ with disjoint universes;
\item a homomorphism $f_{ij} \colon \A_{i} \to \A_{j}$, for every $i, j \in I$ such that $i \leq j$;
\eroman
moreover, $f_{ii}$ is the identity map for every $i \in I$, and if $i \leq j \leq k$, then $f_{ik} = f_{jk} \circ f_{ij}$.
\end{definition}

Let $X$ be a directed system of algebras as above. The \textit{P\l onka sum} of $X$, in symbols $\PL(X)$ or $\PL(\A_{i})_{i \in I}$, is the algebra defined as follows. The universe of $\PL(\A_{i})_{i \in I}$ is the union $\bigcup_{i \in I}A_{i}$. Moreover, for every $n$-ary basic operation $f$ and $a_{1}, \dots, a_{n} \in \bigcup_{i \in I}A_{i}$, we set
\[
f^{\PL(\A_{i})_{i \in I}}(a_{1}, \dots, a_{n}) \coloneqq f^{\A_{j}}(f_{i_{1} j}(a_{1}), \dots, f_{i_{n} j}(a_{n}))
\]
where $a_{1} \in A_{i_{1}}, \dots, a_{n} \in A_{i_{n}}$ and $j = i_{1} \lor \dots \lor i_{n}$.\ 

Observe that if in the above display we replace $f$ by any complex formula $\varphi$ in $n$-variables, we still have that
\[
\varphi^{\PL(\A_{i})_{i \in I}}(a_{1}, \dots, a_{n}) = \varphi^{\A_{j}}(f_{i_{1} j}(a_{1}), \dots, f_{i_{n} j}(a_{n})).
\]
\noindent \textbf{Notation:} Given a formula $\varphi$, we will often write $\varphi^{\PL}$ instead of $\varphi^{\PL(\A_{i})_{i \in I}}$ when no confusion shall occur. 

\vspace{5pt}

The theory of P\l onka sums is strictly related with a special kind of operation:

\begin{definition}\label{def: partition function}
Let $\A$ be an algebra of type $\nu$. A function $\cdot\colon A^2\to A$ is a \emph{partition function} in $\A$ if the following conditions are satisfied for all $a,b,c\in A$, $ a_1 , ..., a_n\in A^{n} $ and for any operation $g\in\nu$ of arity $n\geqslant 1$.
\begin{enumerate}[label=\textbf{P\arabic*}., leftmargin=*]
\item $a\cdot a = a$
\item $a\cdot (b\cdot c) = (a\cdot b) \cdot c $
\item $a\cdot (b\cdot c) = a\cdot (c\cdot b)$
\item $g(a_1,\dots,a_n)\cdot b = g(a_1\cdot b,\dots, a_n\cdot b)$
\item $b\cdot g(a_1,\dots,a_n) = b\cdot a_{1}\cdot_{\dots}\cdot a_n $
\end{enumerate}
\end{definition}

The next result makes explicit the relation between P\l onka sums and partition functions:

\begin{theorem}\cite[Thm.~II]{Plo67}\label{th: Teorema di Plonka}
Let $\A$ be an algebra of type $\nu$ with a partition funtion $\cdot$. The following conditions hold:
\begin{enumerate}
\item $A$ can be partitioned into $\{ A_{i} : i \in I \}$ where any two elements $a, b \in A$ belong to the same component $A_{i}$ exactly when
\[
a= a\cdot b \text{ and }b = b\cdot a.
\]
Moreover, every $A_{i}$ is the universe of a subalgebra $\A_{i}$ of $\A$.
\item The relation $\leq$ on $I$ given by the rule
\[
i \leq j \Longleftrightarrow \text{ there exist }a \in A_{i}, b \in A_{j} \text{ s.t. } b\cdot a =b
\]
is a partial order and $\langle I, \leq \rangle$ is a semilattice. 
\item For all $i,j\in I$ such that $i\leq j$ and $b \in A_{j}$, the map $f_{ij} \colon A_{i}\to A_{j}$, defined by the rule $f_{ij}(x)= x\cdot b$ is a homomorphism. The definition of $f_{ij}$ is independent from the choice of $b$, since $a\cdot b = a\cdot c$, for all $a\in A_i$ and $c\in A_j$.
\item $Y = \langle \langle I, \leq \rangle, \{ \A_{i} \}_{i \in I}, \{ f_{ij} \! : \! i \leq j \}\rangle$ is a directed system of algebras such that $\PL(Y)=\A$.
\end{enumerate}
\end{theorem}

It is worth remarking that the construction of Plonka sums preserves the validity of the so-called \textit{regular identities}, i.e. identities of the form $ \varphi \thickapprox \psi $ such that $\Var(\varphi) = \Var(\psi)$ (for details, see \cite{Romanowska92, Plo67}).

\section{The left variable inclusion companion of a logic}\label{sec: definizione e primi risultati}

The definition of \emph{directed system} can be extended, as follows, to logical matrices:

\begin{definition}\label{Def:Directed-System-Matrices}
A \textit{directed system} of matrices consists of:
\benroman
\item a semilattice $I = \langle I, \lor\rangle$;
\item a family of matrices $\{ \langle\A_{i},F_{i}\rangle \}_{i \in I}$ with disjoint universes;
\item a homomorphism $f_{ij} \colon \A_{i} \to \A_{j}$ such that $f_{ij}[F_{i}] \subseteq F_{j}$, for every $i, j \in I$ such that $i \leq j$;
\eroman
moreover, $f_{ii}$ is the identity map for every $i \in I$, and if $i \leq j \leq k$, then $f_{ik} = f_{jk} \circ f_{ij}$.
\end{definition}

Given directed system of matrices $X$ as above, we set
\[
\PL(X) \coloneqq \langle \PL(\A_{i})_{i \in I}, \bigcup_{i \in I}F_{i}\rangle.
\]
The matrix $\PL(X)$ is the \textit{P\l onka sum} of the directed system of matrices $X$. Given a class $\mathsf{M}$ of matrices, we denote by $\PL(\mathsf{M})$ the class of all P\l onka sums of directed systems of matrices in $\mathsf{M}$. The following observation is a routine computation:

\begin{lemma}\label{Lem:plonka-commutes-with-S-and-P}
$\SSS\PL(\mathsf{M}) \subseteq \PL(\SSS(\mathsf{M}))$ and $\PPP\PL(\mathsf{M}) \subseteq \PL(\PPP(\mathsf{M}))$, for every class of matrices $\mathsf{M}$.
\end{lemma}

\begin{definition}\label{def: regolarizzazione}
Let $\vdash$ be a logic. The \textit{left variable inclusion} companion of $\vdash$ is the relation $\vdash^{l} \subseteq \mathcal{P}(Fm) \times Fm$ defined for every $\Gamma\cup\{\varphi\}\subseteq Fm$ as
\[
\Gamma \vdash^{l} \varphi \Longleftrightarrow \text{ there is }\Gamma' \subseteq \Gamma \text{ s.t. }\Var(\Gamma') \subseteq \Var(\varphi) \text{ and }\Gamma' \vdash \varphi.
\]
\end{definition}

It is immediate to check that $\vdash^{l}$ is indeed a logic and that $\vdash^{l}\subseteq\?\?\?\vdash$. We will often refer to the left variable inclusion of a logic simply as its variable inclusion companion.

\begin{example}\label{ex: PWK}
Let $\vdash$ be propositional classical logic. Then $\vdash^{l}$ is the logic known as Paraconsistent Weak Kleene, \PWK for short, originally introduced in \cite{Kleene}. This logic is equivalently defined, \emph{syntactically}, by imposing the variable inclusion constrain, as in Definition \ref{def: regolarizzazione}, to classical logic or, \emph{semantically} via the so-called \emph{weak Kleene tables} with two of the three truth values as designated (see \cite{Bonzio16,CiuniCarrara}).  
\end{example}

\begin{example}
The left variable inclusion companions of Strong Kleene logic and of  the logic of Paradox (introduced in \cite{Priestfirst}) have been introduced and discussed in \cite{Szmuc}. They are semantically defined, by adding a nonsensical truth value to the (single) matrix inducing Strong Kleene and the logic of Paradox, respectively. 

\end{example}

In \cite{Bonzio16}, it is shown that an algebraic semantics for \PWK is obtained via P\l onka sums of Boolean algebras. We shall show that this idea can be generalized to the variable inclusion companion of any logic $\vdash$.

\begin{lemma}\label{lem:soundness}
Let $\vdash$ be a logic and $X$ be a directed system of models of $\vdash$. Then $\PL(X)$ is a model of $\vdash^{l}$.
\end{lemma}

\begin{proof}
Suppose that $\Gamma \vdash^{l}\varphi$ and consider a homomorphism $v \colon \Fm \to \PL(\A_{i})_{i \in I}$ such that $v[\Gamma] \subseteq \bigcup_{i \in I}F_{i}$. By the definition of $\vdash^{l}$, there exists $\Delta\subseteq\Gamma$ such that $\Var(\Delta)\subseteq\Var(\varphi)$ and $\Delta\vdash\varphi$. Consider an enumeration $\Var(\varphi) = \{ x_{1}, \dots, x_{n}\}$. There are $i_{1}, \dots, i_{n} \in I$ such that $v(x_{1}) \in A_{i_{1}}, \dots, v(x_{n}) \in A_{i_{n}}$. We set $j \coloneqq i_{1} \lor \dots \lor i_{n}$.

Now, consider a homomorphism $g \colon \Fm \to \A_{j}$ such that
\[
g(x_{m}) = f_{i_{m}j}(v(x_{m})) \text{, for every }m \leq n.
\]
We claim that $g[\Delta] \subseteq F_{j}$. To prove this, consider an arbitrary formula $\delta \in \Delta$. Since $\Var(\Delta) \subseteq \{ x_{1}, \dots, x_{n} \}$, we can assume that $\Var(\delta) = \{x_{m_1}, \dots, x_{m_k}\}\subseteq \{ x_{1}, \dots, x_{n} \}$ for some $k \leq n$. Set $l\coloneqq i_{m_1} \lor \dots \lor i_{m_k}$. From the definition of $\PL(X)$ we have that 
\[
v(\delta) = \delta^{\PL}(v(x_{m_1}), \dots, v(x_{m_k})) = \delta^{\A_{l}}(f_{i_{m_1}l}(v(x_{m_1})), \dots, f_{i_{m_k}l}(v(x_{m_k}))). 
\]
Since $v(\delta)\in \bigcup_{i\in I}F_i$, this implies that 
\begin{equation}\label{Eq:displayplonka}
\delta^{\A_{l}}(f_{i_{m_1}l}(v(x_{m_1})), \dots, f_{i_{m_k}l}(v(x_{m_k}))) \in F_{l}.
\end{equation}
Now observe that $l \leq j$. Therefore there is a homomorphism $f_{lj} \colon \A_{l} \to \A_{j}$ such that $f_{lj}[F_{l}] \subseteq F_{j}$. Together with (\ref{Eq:displayplonka}), this implies that
\begin{align*}
g(\delta) =&\delta^{\A_{j}}(f_{i_{m_1}j}(v(x_{m_1})), \dots, f_{i_{m_k}j}(v(x_{m_k})))\\
 = &\delta^{\A_{j}}(f_{lj} \circ f_{i_{m_1}l}(v(x_{m_1})), \dots, f_{lj} \circ f_{i_{m_k}l}(v(x_{m_k})))\\
=&f_{lj}\delta^{\A_{l}}(f_{i_{m_1}l}(v(x_{m_1})), \dots, f_{i_{m_k}l}(v(x_{m_k})))\\
\in &\?\?  f_{lj}[F_{l}]  \subseteq F_{j}.
\end{align*}
This establishes our claim.

Recall that $\Delta \vdash \varphi$. Since $\langle \A_{j}, F_{j}\rangle$ is a model of $\vdash$ and by the claim $g[\Delta] \subseteq F_{j}$, we conclude that $g(\varphi) \in F_{j}$. But this means that
\begin{align*}
v (\varphi) &= \varphi^{\PL}(v(x_{1}), \dots, v(x_{n}))\\
&= \varphi^{\A_{j}}(f_{i_{1}j}(v(x_{1})), \dots, f_{i_{n}j}(v(x_{n})))\\
&= g(\varphi) \in \?\? F_{j} \subseteq \bigcup_{i \in I}F_{i}.
\end{align*}
Hence we conclude that $\PL(X)$ is a model of $\vdash^{l}$ as desired.
\end{proof}

Recall that $\mathbf{1}$ is the trivial algebra. The following construction originates in \cite{Lakser72}. Given an algebra $\A$, there is always a directed system of algebras given by $\A$ and $\boldsymbol{1}$ equipped with the identity endomorphisms and the unique homomorphism $f \colon \A \to \boldsymbol{1}$. We denote by $\A \oplus \boldsymbol{1}$ the P\l onka sum of this directed system. Observe that $\A \oplus \boldsymbol{1}$ is the algebra with universe $A \cup \{ 1 \}$ and basic operations $f$ defined as follows:
\[
f^{\A\oplus \boldsymbol{1}}(a_{1}, \dots, a_{n})\coloneqq \left\{ \begin{array}{ll}
 f^{\A}(a_{1}, \dots, a_{n})& \text{if $a_{1}, \dots, a_{n} \in A$}\\
   1 & \text{otherwise.}\\
  \end{array} \right.  
\]

Observe that the above construction can be lifted to matrices. More precisely, given an arbitrary matrix $\langle \A, F\rangle$, there is always a directed system of matrices given by $\langle \A, F\rangle$ and $\langle \boldsymbol{1}, \{ 1 \} \rangle$ equipped with the identity endomophisms and the unique homomorphism $f \colon \A \to \boldsymbol{1}$. The P\l onka sum of this system is the matrix $\langle\A \oplus \boldsymbol{1}, F \cup \{ 1 \} \rangle$.

\begin{theorem}\label{thm:completeness}
Let $\vdash$ be a logic and $\mathsf{M}$ be a class of matrices containing $\langle \boldsymbol{1}, \{ 1 \} \rangle$. If $\vdash$ is complete w.r.t.\ $\mathsf{M}$, then $\vdash^{l}$ is complete w.r.t.\ $\PL(\mathsf{M})$.
\end{theorem}

\begin{proof}
In the light of Lemma \ref{lem:soundness} it will be enough to show that if $\Gamma \nvdash^{l}\varphi$, then $\Gamma \nvdash_{\PL(\mathsf{M})} \varphi$. To this end, suppose that $\Gamma \nvdash^{l}\varphi$. Define
\begin{align*}
\Gamma^{+}&\coloneqq \{ \gamma \in \Gamma : \Var(\gamma) \subseteq \Var(\varphi)\}\\
\Gamma^{-}& \coloneqq \{ \gamma \in \Gamma : \Var(\gamma)\nsubseteq \Var(\varphi)\}.
\end{align*}
Clearly $\Gamma = \Gamma^{+} \cup \Gamma^{-}$. Since $\Gamma \nvdash^{l}\varphi$, we know that $\Gamma^{+} \nvdash \varphi$. Together with the fact that $\vdash$ is complete w.r.t.\ $\mathsf{M}$, this implies that there exists a matrix $\langle \A, F\rangle \in \mathsf{M}$ and a homomorphism $v \colon \Fm \to \A$ such that $v[\Gamma^{+}] \subseteq F$ and $v(\varphi) \notin F$.

Since $\langle \A, F\rangle, \langle \boldsymbol{1}, \{ 1 \} \rangle \in \mathsf{M}$, we have that $\langle \A \oplus \boldsymbol{1}, F \cup \{ 1 \} \rangle \in \PL(\mathsf{M})$. Now, consider the homomorphism $g \colon \Fm \to \A \oplus \boldsymbol{1}$ defined for every variable $x \in \Var$ as follows:
\[
g(x)\coloneqq \left\{ \begin{array}{ll}
 v(x)& \text{if $x \in \Var(\varphi)$}\\
   1 & \text{otherwise.}\\
  \end{array} \right.  
\]
From the definition of $\A \oplus \boldsymbol{1}$ it follows that
\begin{align*}
g[\Gamma^{-}] &\subseteq \{ 1 \} \subseteq F \cup \{ 1 \}\\
g(\gamma) &= v(\gamma) \text{ for every }\gamma \in \Gamma^{+} \cup \{ \varphi \}.
\end{align*}
Together with the fact that $v[\Gamma^{+}] \subseteq F$ and $v(\varphi) \notin F$, this implies that
\[
g[\Gamma] = g[\Gamma^{+} \cup \Gamma^{-}] \subseteq F \cup \{ 1 \} \text{ and } g(\varphi) \notin F \cup \{ 1 \}.
\]
Hence we conclude that $\Gamma \nvdash_{\PL(\mathsf{M})} \varphi$ as desired.
\end{proof}
\begin{corollary}\label{cor:completeness}
Let $\vdash$ be a logic. The variable inclusion companion $\vdash^{l}$ is complete w.r.t.\ any of the following classes of matrices:
\[
\PL(\Mod(\vdash)), \quad \PL(\Modstar(\vdash)),\quad \PL(\ModS(\vdash)).
\]
\end{corollary}

\begin{proof}
Observe that $\vdash$ is complete w.r.t.\ any of the classes $ \Mod(\vdash)$, $\Modstar(\vdash)$, $\ModS(\vdash) $. Moreover any of these classes contains the (trivial) matrix $\langle \boldsymbol{1}, \{ 1 \} \rangle$. Thus we can apply Theorem \ref{thm:completeness}.
\end{proof}

\section{Logics with a partition function and axiomatizations}\label{sec: logics with p.f.}

\begin{definition}\label{def: logic with p-func.}
A logic $\vdash$ has a \emph{partition function} if there is a formula $x\cdot y$, in which the variables $x$ and $y$ really occur, such that $x \vdash x\cdot y$ and the operation $\cdot^{\A}$ is a partition function for every $\A\in\Alg(\vdash)$.
 In this case, $x \cdot y$ is a \emph{partition function} for $\vdash$.
\end{definition}

\begin{remark}\label{Rem:SyntacticDefinitionPartitionFunction}
By Lemma \ref{lemma su equazioni p-function}, the above Definition can be rephrased 
in purely logical terms, by requiring that $x\vdash x\cdot y$ and that
\[
\varphi(\epsilon,\vec{z}\?\?)\sineq\varphi(\delta,\vec{z}\?\?)\text{ for every formula }\varphi(v,\vec{z}\?\?), 
\]
for every identity of the form $\epsilon \thickapprox \delta$ in $\mathbf{P1.}, \dots,  \mathbf{P5.}$
\qed
\end{remark}

\begin{example}
Logics with a partition function abound in the literature. Indeed, it is easy to check that the term $x\cdot y\coloneqq x\land(x\lor y)$ is a partition function for every logic $\vdash$ such that every algebra in $\Alg(\vdash)$ has a lattice reduct. Such examples include all modal and substructural logics. On the other hand, $x\cdot y\coloneqq (y \to y) \to x$ is a partition function  for all logics $\vdash$ such that $\Alg(\vdash)$ has a Hilbert algebra reduct \cite{Di65}.
\qed
\end{example}

Remarkably, the presence of a partition function is inherited by the variable inclusion companion of a logic.

\begin{lemma}\label{lemma: p.f. della regolarizzazione}
Let $\vdash$ be a logic. The operation $\cdot$ is a partition function for $\vdash$ if and only if it is a partition function for $\vdash^{l}$.
\end{lemma}
\begin{proof}
From Remark \ref{Rem:SyntacticDefinitionPartitionFunction} the fact that $\cdot$ is a partition function for $\vdash$ is witnessed by the validity of some inferences $\varphi\vdash\psi$ such that $\Var(\varphi)\subseteq\Var(\psi)$. Hence these inferences also hold in $\vdash^{l}$. With another application of Remark \ref{Rem:SyntacticDefinitionPartitionFunction} we conclude that $\cdot$ is a partition function for $\vdash^{l}$. 
 
The other direction follows from the inclusion $\vdash^{l}\subseteq\?\?\?\vdash$.
\end{proof}

The following result is the generalization of Theorem \ref{th: Teorema di Plonka} to the setting of logical matrices. 

\begin{theorem}\label{th: plonka sum of matrices}
Let $\vdash$ be a logic with a partition function $\cdot$, and $\langle\A,F\rangle$ be a model of $\vdash$ such that $\A\in\Alg(\vdash)$. Conditions (1-4) of Theorem \ref{th: Teorema di Plonka} hold. Moreover, setting $F_{i} \coloneqq F\cap A_{i}$ for every $i \in I$, the triple
\[
X=\langle \langle I,\leq\rangle , \{ \langle \A_{i},F_{i}\rangle\}_{i\in I}, \{ f_{ij} \! : \! i \leq j \}\rangle
\]
is a directed system of matrices such that $\PL(X)=\langle\A,F\rangle$.
\end{theorem}
\begin{proof}
In the light of Theorem \ref{th: Teorema di Plonka}, it will be enough to show that $f_{ij}[F_{i}]\subseteq F_{j}$ for every $i, j \in I$ such that $i\leq j$. To this end, consider $a\in F_{i}$ and $b \in A_{j}$ with $i\leq j$. Since $\cdot$ is a partition function for $\vdash$, we have $x\vdash x\cdot y$. Together with the fact that $\langle\A, F\rangle\in\Mod(\vdash)$ and $a \in F$, this implies that $a\cdot^{\A}b\in F$. Observe that $a\cdot^{\A}b \in A_{j}$ by (2) in Theorem \ref{th: Teorema di Plonka} and, therefore, that $a\cdot^{\A}b \in F_{j}$. Hence, by (3), we have that $f_{ij}(a) = a\cdot^{\A}b\in F_{j} $.
\end{proof}

\begin{definition}
Let $\vdash$ be a logic with a partition function $\cdot$, and $\langle\A,F\rangle$ be a model of $\vdash$ such that $\A\in\Alg(\vdash)$. The \emph{P\l onka fibers}
 of $\langle \A, F \rangle$ are the matrices $\{ \langle \A_{i},F_{i}\rangle\}_{i\in I}$ given in the above result.
\end{definition}

\begin{lemma}\label{lm: regolarizzazione con p.f.}
Let $\vdash$ be a finitary logic with partition function $\cdot$ and $\pair{\A, F}\in \Mod(\vdash^{l})$, with $\A\in\Alg(\vdash^{l})$. Then, the P\l onka fibers of $\pair{\A, F} $ are models of $\vdash$.
\end{lemma}
\begin{proof}
Let $\pair{\A_i , F_i}$ be a P\l onka fiber of $\pair{\A, F} $ and $\Gamma\vdash \varphi$, with $\Gamma$ a finite set. Then consider a homomorphism $v\colon\Fm\to\A_i$ such that $v[\Gamma]\subseteq F_i$. Then, there are cases: either $\Gamma $ is empty or not. First, suppose $\Gamma=\emptyset$. Then clearly $\emptyset\vdash^{l}\varphi$. Since $\A_i$ is a subalgebra of $\A$ and $\langle \A, F \rangle$ is a model of $\vdash^{l}$, this implies that $v(\varphi) \in F \cap A_i = F_i$. Then consider the case where $\Gamma$ is non\added[id=Stef]{-}void. Then there are $\gamma_{1},\dots,\gamma_{n}\in Fm$ such that $\Gamma=\{\gamma_{1},\dots, \gamma_n\}$. Since $\cdot$ is a partition function, we have $x \vdash x \cdot y$. In particular, this implies that $\varphi\vdash\varphi\cdot(\gamma_{1}\cdot(\gamma_{2}\cdot \dots(\gamma_{n-1}\cdot\gamma_{n})\dots))$. Then $\Gamma\vdash \varphi\cdot(\gamma_{1}\cdot(\gamma_{2}\cdot \dots(\gamma_{n-1}\cdot\gamma_{n})\dots))$. Since the variable inclusion constraint holds for this inference, we obtain that
\[
\Gamma\vdash^{l} \varphi\cdot(\gamma_{1}\cdot(\gamma_{2}\cdot \dots(\gamma_{n-1}\cdot\gamma_{n})\dots)).
\]
Since $\A_i$ is a subalgebra of $\A$ and $\langle \A, F \rangle$ is a model of $\vdash^{l}$, this implies that 
\[
v(\varphi\cdot(\gamma_{1}\cdot(\gamma_{2}\cdot \dots(\gamma_{n-1}\cdot\gamma_{n})\dots)))\in A_i\cap F = F_{i}.
\]
Since $v(\varphi)$ and $v(\gamma_{1}\cdot(\gamma_{2}\cdot \dots(\gamma_{n-1}\cdot\gamma_{n})\dots)))$ belong to $A_i$, this implies that
\begin{align*}
v(\varphi) &= v(\varphi) \cdot v(\gamma_{1}\cdot(\gamma_{2}\cdot \dots(\gamma_{n-1}\cdot\gamma_{n})\dots)))\\
&=v(\varphi\cdot(\gamma_{1}\cdot(\gamma_{2}\cdot \dots(\gamma_{n-1}\cdot\gamma_{n})\dots)))
\end{align*}
and, therefore, that $v(\varphi) \in F_i$, as desired.
\end{proof}

By a \textit{Hilbert-style calculus with finite rules} we understand a (possibly infinite) set of Hilbert-style rules, each of which has finitely many premises.

\begin{definition}\label{def: hilbert calc per regolarizz.}
Let $\mathcal{H}$ be a Hilbert-style calculus with finite rules that determines a logic $\vdash$ with a partition function $\cdot$. Let $\mathcal{H}^{l}$ be the Hilbert-style calculus given by the following rules:

\begin{align}
\emptyset &\rhd \psi\tag{H1}\label{Eq:Axiom1} \\
\gamma_{1}, \dots, \gamma_{n} &\rhd \varphi\cdot(\gamma_{1}\cdot(\gamma_{2}\cdot \dots(\gamma_{n-1}\cdot\gamma_{n})\dots)) \tag{H2}\label{Eq:Axiom2}\\
x &\rhd x\cdot y\tag{H3}\label{Eq:Axiom3}\\
\chi(\epsilon, \vec{z}\?\?) \?\?\?\lhd&\rhd \chi(\delta, \vec{z}\?\?)\tag{H4}\label{Eq:Axiom4}
\end{align}
for every
\begin{enumerate}[(i)]
\item $\emptyset \rhd \psi$ rule in $\mathcal{H}$;
\item $\gamma_{1}, \dots, \gamma_{n}\rhd \varphi$ rule in $\mathcal{H}$;
\item equation $\epsilon \thickapprox \delta$ in the definition of partition function, and formula $\chi(v, \vec{z}\?\?)$.
\end{enumerate}
\end{definition}

\begin{theorem}\label{theor:HilbertCalculus}
Let $\vdash$ be a logic with partition function $\cdot$ defined by a Hilbert-style calculus with finite rules $\mathcal{H}$. Then $\mathcal{H}^{l}$ is a complete Hilbert-style calculus for $\vdash^{l}$. 
\end{theorem}
\begin{proof}
Let $\vdash_{\mathcal{H}^{l}}$ be the logic determined by $\mathcal{H}^{l}$. We begin by showing that $\vdash_{\mathcal{H}^{l}} \?\?\subseteq\?\? \vdash^{l}$. It will be sufficient to show that every rule in $\mathcal{H}^{l}$ holds in $\vdash^{l}$. This is clear for (\ref{Eq:Axiom1}). Moreover, the rules (\ref{Eq:Axiom3}, \ref{Eq:Axiom4}) are valid in $\vdash^{l}$, because $\cdot$ is a partition  function for $\vdash^{l}$ by Lemma \ref{lemma: p.f. della regolarizzazione}. It only remains to prove that (\ref{Eq:Axiom2}) holds in $\vdash^{l}$. To this end, consider a rule $\gamma_{1}, \dots, \gamma_{n} \rhd \varphi$ in $\mathcal{H}$. Clearly we have that $\gamma_{1}, \dots, \gamma_{n} \vdash \varphi$. Since $\cdot$ is a partition function for $\vdash$, we have $x \vdash x\cdot y$. In particular, $\varphi \vdash \varphi\cdot(\gamma_{1}\cdot(\gamma_{2}\cdot \dots(\gamma_{n-1}\cdot\gamma_{n})\dots))$. Hence we conclude that
\[
\gamma_{1}, \dots, \gamma_{n} \vdash^{l} \varphi\cdot(\gamma_{1}\cdot(\gamma_{2}\cdot \dots(\gamma_{n-1}\cdot\gamma_{n})\dots)),
\]
as desired. 

To prove $\vdash^{l} \?\? \subseteq \?\? \vdash_{\mathcal{H}^{l}}$, we reason as follows. Consider  $\langle\A, F\rangle\in\ModS(\vdash_{\mathcal{H}^{l}})$. Observe that clearly $\A \in \Alg(\vdash_{\mathcal{H}^{l}})$. Moreover, $\cdot$ is a partition function in $\vdash_{\mathcal{H}^{l}}$ by Remark \ref{Rem:SyntacticDefinitionPartitionFunction} and (\ref{Eq:Axiom3},\ref{Eq:Axiom4}). Hence we can apply Theorem \ref{th: plonka sum of matrices}, obtaining that $\langle \A, F\rangle = \PL(X)$, where $X$ is the directed system of matrices $\langle I, \{ \langle \A_{i}, F_{i}\rangle\}_{i\in I}, \{ f_{ij} \! : \! i \leq j \}\rangle$ given in the statement of Theorem \ref{th: plonka sum of matrices}. Thanks to the rules of $\mathcal{H}^{l}$ we can replicate the construction in the proof of Lemma \ref{lm: regolarizzazione con p.f.} obtaining that each fiber $\pair{\A_i, F_i}$ is a model of $\vdash$. This observation, together with the fact that $\langle \A, F\rangle = \PL(X)$ and Corollary \ref{cor:completeness}, implies that $\langle \A, F\rangle$ is a model of $\vdash^{l}$. Hence we conclude that $\ModS(\vdash_{\mathcal{H}^{l}}) \subseteq \Mod(\vdash^{l})$. This implies that $\vdash^{l} \?\? \subseteq \?\? \vdash_{\mathcal{H}^{l}}$.
 \end{proof}
 
The proof of the above result establishes the following:

 \begin{corollary}\label{cor: ModSu incluso nelle somme di Plonka dei modelli}
If $\vdash$ is a finitary logic with a partition function, then $\ModS(\vdash^{l}) \subseteq \PL (\Mod(\vdash))$.
\end{corollary}

\begin{example}\label{calcolo alla hilbert finito per pwk}
A Hilbert-style calculus for \PWK is axiomatized, following Definition \ref{def: hilbert calc per regolarizz.}, as follows ($\varphi\to\psi$ is a shorthand for $\neg\varphi\lor\psi$):
%\begin{align}
%\emptyset &\rhd \psi\tag{H1}%\label{Eq:Axiom1} 
%\end{align}

\begin{align}
\emptyset &\rhd(\varphi\lor\varphi)\to\varphi\tag{A1} \\
\emptyset &\rhd\varphi\to(\varphi\lor\psi)\tag{A2} \\
\emptyset &\rhd(\varphi\lor\psi)\to(\psi\lor\varphi)\tag{A3} \\
\emptyset &\rhd(\varphi\to\psi)\to((\gamma\lor\varphi)\to(\gamma\lor\psi))\tag{A4} \\
\emptyset &\rhd (\varphi\land\psi)\to\lnot(\lnot\varphi\lor\lnot\psi)\tag{A5}\\
\emptyset &\rhd\lnot(\lnot\varphi\lor\lnot\psi)\to(\varphi\land\psi)\tag{A6} \\
\varphi,\varphi\to\psi &\rhd \psi\land(\psi\lor (\varphi\land(\varphi\lor (\varphi\to\psi)))) \tag{R1} \\
\varphi &\rhd \varphi\land (\varphi\lor\psi)\tag{R2}\\
\chi(\epsilon, \vec{z}\?\?) \?\?\?\lhd&\rhd \chi(\delta, \vec{z}\?\?)\tag{R*}%\label{Eq:Axiom4}
\end{align}

Notice that Axioms (A1)--(A6), together with the rule of \emph{Modus Ponens}, provide a Hilbert-style calculus for propositional classical logic. (R1) and (R2) are obtained by setting $x\cdot y\coloneqq x\land (x\vee y)$ as partition function for classical logic. Note, moreover, that (R*) is in fact a rule scheme, summarizing an infinity of rules.
\qed
\end{example}

\section{Suszko reduced models of $\vdash^{l}$}\label{sec: Suszko red mod}

In this section we investigate the structure of the Suszko reduced models $\ModS(\vdash^{l})$ of the variable inclusion companion $\vdash^{l}$ of a logic $\vdash$ (with partition function).  To this end, we rely on the following technical observation:

\begin{lemma}\label{lem:basic-facts}
Let $\vdash$ be a logic with a partition function $\cdot$, and  $X=\langle\?\?\langle I,\leq\rangle , \{ \langle \A_{i},F_{i}\rangle\}_{i\in I}, \{ f_{ij} \! : \! i \leq j \}\rangle$ a directed system of models of $\vdash$. Given an upset $J \subseteq I$, we define for every $i \in I$,
 \[
G_{i}\coloneqq \left\{ \begin{array}{ll}
A_{i}& \text{if $i \in J$}\\
   F_{i} & \text{otherwise.}\\
  \end{array} \right.  
\]
Then $\bigcup_{i \in I}G_{i}$ is a $\vdash^{l}$-filter on $\PL(\A_{i})_{i \in I}$.

\end{lemma}

\begin{proof}
It is clear that the matrices $\{ \langle \A_{i}, G_{i}\rangle : i \in I \}$ give naturally rise to a directed system of matrices, when equipped with the homomorphisms in  $X$. Moreover, by assumption each $\langle \A_{i}, G_{i}\rangle$ is a model of $\vdash$. Thus $\bigcup_{i \in I}G_{i}$ is a $\vdash^{l}$-filter on $\PL(\A_{i})_{i \in I}$ by Lemma \ref{lem:soundness}.
\end{proof}

The following result identifies the P\l onka sums of matrices in $\ModS(\vdash)$ that belong to $\ModS(\vdash^{l})$.

\begin{theorem}\label{th: caratterizzazione Plonka suszko ridotta}
Let $\vdash$ be a logic with a partition function $\cdot$, and let $X=\langle \langle I,\leq\rangle , \{ \langle \A_{i},F_{i}\rangle\}_{i\in I}, \{ f_{ij} \! : \! i \leq j \}\rangle$ be a directed system of matrices in $\ModS(\vdash)$. The following conditions are equivalent:
\benroman
\item $\PL(X)\in\ModS(\vdash^{l})$.
\item For every $n, i \in I$ such that $\langle \A_{n}, F_{n}\rangle$ is trivial and $n<i$, there exists $j \in I$ s.t. $n\leq j, i\nleq j$ and $\A_{j}$ is non-trivial.
\eroman
\end{theorem} 

 \begin{proof}
 (i)$\Rightarrow$(ii): Suppose that $\PL(X)\in\ModS(\vdash^{l})$, and consider $n, i \in I$ such that $\langle \A_{n}, F_{n}\rangle$ is trivial and $n<i$. The fact that $\langle \A_{n}, F_{n}\rangle$ is both trivial and belongs to $\ModS(\vdash)$ implies that $\A_{n}$ is the trivial algebra. Then $\pair{\A_{n} , F_{n}}=\pair{\mathbf{1}, \{1\}}$. Moreover, set $a \coloneqq f_{ni}(1)$. Since $n < i$, we know that $a \ne 1$.  Together with the fact that $\PL(X)\in\ModS(\vdash^{l})$, this implies that there is a $\vdash^{l}$-filter $G$ of $\PL(\A_{i})_{i \in I}$ such that $\bigcup_{i \in I}F_{i} \subseteq G$ and $\langle a, 1 \rangle \notin \Leibniz^{\PL(\A_{i})_{i \in I}} G$. Thus, by Lemma \ref{lem:polynomial charact. Suszko cong.}, there is a formula $\varphi(x, \vec{z})$ and elements $\vec{c} \in \bigcup_{i \in I}A_{i}$ such that
 \begin{equation}\label{eq:zz1}
 \varphi^{\PL}(a, \vec{c}\?\?) \in G \Longleftrightarrow \varphi^{\PL}(1, \vec{c}\?\?) \notin G.
 \end{equation}
\noindent 
We can assume w.l.o.g.\ that all the elements in the sequence $\vec{c}$ belong to the same component $A_{k}$ of the P\l onka sum $\PL(\A_{i})_{i \in I}$.\footnote{More precisely, if $\vec{c} = c_{1}, \dots, c_{m}$ and $c_{1} \in A_{p_{1}}, \dots, c_{m} \in A_{p_{m}}$, then we set $k \coloneqq p_{1} \lor \dots \lor p_{m}$ and replace $c_i$ by $f_{p_{i}k}(c_i)$.}
 
We claim that indeed $\varphi^{\PL}(1, \vec{c}\?\?) \notin G$. Suppose the contrary towards a contradiction. Then $\varphi^{\PL}(1, \vec{c}\?\?) \in G$.  First observe that
\begin{align}
\varphi^{\PL}(a, \vec{c}\?\?) &= \varphi^{\A_{i \lor k}}(f_{i, i \lor k}(a), f_{k, i \lor k}(\vec{c}\?\?))\label{Eq:mmm0}\\
&= \varphi^{\A_{i \lor k}}(f_{i, i \lor k} \circ f_{n, i}(1), f_{k, i \lor k}(\vec{c}\?\?)) \label{Eq:mmm1}\\
&=\varphi^{\A_{i \lor k}}(f_{n \lor k, i \lor k} \circ f_{n, n \lor k}(1), f_{n \lor k, i \lor k} \circ f_{k, n \lor k}(\vec{c}\?\?))\label{Eq:mmm2}\\
&=f_{n \lor k, i \lor k} \?\? \varphi^{\A_{n \lor k}}(f_{n, n \lor k}(1), f_{k, n \lor k}(\vec{c}\?\?))\label{Eq:mmm3}\\
&= f_{n \lor k, i \lor k} \?\?\varphi^{\PL}(1, \vec{c}\?\?)\label{Eq:mmm4}\\
&=f_{n \lor k, i \lor k} (\varphi^{\PL}(1, \vec{c}\?\?))\cdot^{\A_{i \lor k}} f_{i, i \lor k}(a)\label{Eq:mmm5}\\
&= \varphi^{\PL}(1, \vec{c})\cdot^{\PL} a\label{Eq:mmm6}\\
&\in G.\label{Eq:mmm7}
\end{align}
The above equalities are justified as follows: (\ref{Eq:mmm2}) is a consequence of the fact that $X$ is a directed system of matrices and that $n\vee k\leq i\vee k$ (since $n\leq i$), (\ref{Eq:mmm5}) follows from the fact that $x\cdot^{\PL} y$ is the projection on the first component on the algebra $\A_{i \lor k}$. Condition (\ref{Eq:mmm7}) follows from the fact that $\varphi^{\PL}(1, \vec{c}\?\?)\in G$, $G$ is a $\vdash^{l}$-filter and, by Lemma \ref{lemma: p.f. della regolarizzazione} $\cdot$ is a partition function for $\vdash^{l}$, hence $x\vdash^{l}x\cdot y$.
Hence we have that $\varphi^{\PL}(a, \vec{c}\?\?),\varphi^{\PL}(1,\vec{c}\?\?) \in G$, which contradicts \eqref{eq:zz1}, establishing the claim.

From the claim and (\ref{eq:zz1}) we get that $\varphi^{\PL}(a, \vec{c}\?\?) \in G$ and $\varphi^{\PL}(1, \vec{c}\?\?) \notin G$. Set $j \coloneqq n\vee k$ and $m \coloneqq k\vee i$. We claim that $j$ is such that: (A) $n\leq j$, (B) $\A_j$ is non trivial and (C) $i\nleq j$. We proceed to prove (A, B, C).

(A): Since $j=n\vee k$, we have that $n\leq j$. 

%Suppose by contradiction that $n=l$, i.e.\ that $k \leq n$. Recall that $F_{n} = A_{n}$, since $\langle \A_{n}, F_{n}\rangle$ is trivial. Keeping this in mind, we have that
%\[
%\varphi^{\PL}(b,\vec{c}\?\?)= \varphi^{\A_{n}}(b,\vec{c}\?\?) \in A_{n} =  F_{n} \subseteq \bigcup_{i \in I}F_{i} \subseteq G,
%\]
%against the fact that $\varphi(b,\vec{c}\?\?)\notin G$.\\
(B): Observe that 
\[
\varphi^{\PL}(1,\vec{c}\?\?) = \varphi^{\A_j}(f_{nj}(1),f_{kj}(\vec{c}\?\?))\in A_j.
\]
Together with $\varphi^{\PL}(1,\vec{c}\?\?)\notin G$, this implies that $\varphi^{\PL}(1,\vec{c}\?\?)\in A_j\smallsetminus G $. 
% we have $\varphi(b,\vec{c}\?\?) = \varphi^{\A_l}(b,\vec{c}\?\?)\in A_l \smallsetminus G$. 

On the other hand, since $F_n = A_n$, we have that
\[
f_{nj}(1) \in f_{nj}[F_{n}] \subseteq F_{j} \subseteq A_{j} \cap G.
\]
Thus both $A_{j} \cap G$ and $A_{j} \smallsetminus G$ are non-empty. We conclude that $\A_{j}$ is non-trivial.

(C): Suppose, by contradiction, that $i\leq j$. In particular, this implies that $m= j$ (indeed, $i\leq j = n\lor k$, thus $i\lor k\leq n\lor k$, i.e. $m\leq j$; on the other hand, since $n < i$ then $n\lor k \leq i\lor j$, i.e. $j\leq m$). Therefore we have that  
 \begin{align}
 \varphi^{\PL}(1,\vec{c}\?\?) &= \varphi^{\A_j}(f_{nj}(1),f_{kj}(\vec{c}\?\?))\label{p1} && 
\\ &= \varphi^{\A_j}(f_{ij} \circ f_{ni}(1),f_{kj}(\vec{c}\?\?))\label{p2} && 
\\ & = \varphi^{\A_j}(f_{ij}(a),f_{kj}(\vec{c}\?\?))\label{p3} && 
\\ & = \varphi^{\A_m}(f_{im}(a),f_{km}(\vec{c}\?\?))\label{p4} && 
\\ & = \varphi^{\PL}(a,\vec{c}\?\?)\in G\label{p5}.&&
\end{align}
The above equalities are justified as follows. \eqref{p2} follows from the fact that $i\leq m=j$. \eqref{p3} is a consequence of $a=f_{ni}(1)$. \eqref{p4} from $j=m$ and \eqref{p5} from $m=i\vee k$. This establishes the above equalities, yielding that $\varphi^{\PL}(1,\vec{c}\?\?) \in G$. But this contradicts the fact that $\varphi^{\PL}(1,\vec{c}\?\?) \notin G$. 

Hence (A), (B) and (C) hold establishing our claim. In particular, this implies that $j\in I$ satisfies the condition in the statement.
\vspace{10pt}
%and concludes the first implication. \\

 (ii)$\Rightarrow$(i):  By Lemma \ref{lem:soundness} we know that $\PL(X)$ is a model of $\vdash^{l}$. It only remains to prove that it is Suszko reduced.  To this end, let $\theta$ be the Suszko congruence of $\PL(X)$. 
 
Observe that, in order to prove that $\theta$ is the identity, it will be enough to show that it does not identify distinct elements in components of the P\l onka sum which are comparable with respect to the order $\leq$. To prove this, suppose indeed that $\theta$ does not identify different elements in components of the P\l onka sum which are comparable. Then consider two different elements $a, b \in A=\bigcup_{i\in I} A_i$. There exist $i, j \in I$ such that $a \in A_{i}$ and $b \in A_{j}$. If $i$ and $j$ are comparable, then by assumption $\langle a, b\rangle\notin\theta$. Then consider the case where $i$ and $j$ are incomparable. Set $k \coloneqq i \lor j$. Clearly we have that $i, j < k$. In particular, we have that $b\cdot b = b \in A_{j}$ and $a\cdot b \in A_{k}$ and, therefore, $b\cdot^{\PL} b \ne a\cdot^{\PL} b$. Since $j$ and $k$ are comparable, this implies that $\langle b\cdot^{\PL} b, a\cdot^{\PL} b\rangle\notin\theta$. In particular, this means that $\langle a, b\rangle\notin\theta$ as well. As a consequence we conclude that $\theta$ is the identity.
 
By the above observation, to prove that $\theta$ is the identity, it will be enough to show that it does not identify elements in components of the P\l onka sum $\PL(X)$ which are comparable with respect to $\leq$. To this end, consider two different elements $a, b \in A$ such that $a \in A_{i}$ and $b \in A_{j}$ with $i \leq j$. We have two cases: either $i = j$ or $i < j$.

First consider the case where $i = j$, that is $a, b \in A_{i}$. By assumption, we have that $\langle \A_{i}, F_{i}\rangle \in \ModS(\vdash)$. Therefore we can assume w.l.o.g.\ that there is a $\vdash$-filter $G_{i}$ on $\A_{i}$ such that $F_{i} \subseteq G_{i}$, some elements $\vec{c} \in A_{i}$, and a formula $\varphi(x, \vec{z})$ such that $\varphi^{\A_{i}}(a, \vec{c}\?\?) \in G_{i}$ and $\varphi^{\A_{i}}(b, \vec{c}\?\?) \notin G_{i}$. For every $l \ne i$, define
 \[
G_{l}\coloneqq \left\{ \begin{array}{ll}
A_{l}& \text{if $i \leq l$}\\
   F_{l} & \text{otherwise.}\\
  \end{array} \right.  
\]
 An analogous argument to the one described in the proof Lemma \ref{lem:basic-facts} shows that $G \coloneqq \bigcup_{i \in I}G_{i}$ is a $\vdash^{l}$-filter on $\PL(\A_{i})_{i \in I}$. Moreover, observe that
\begin{align*}
\varphi^{\PL}(a, \vec{c}) &= \varphi^{\A_{i}}(a, \vec{c}) \in G\\
\varphi^{\PL}(b, \vec{c}) &= \varphi^{\A_{i}}(b, \vec{c}) \notin G.
\end{align*}
We conclude that $\langle a, b \rangle\notin\theta$.% does not belong to the Suszko congruence of  $\PL(X)$.

Then we consider the case where $i < j$. We have cases: either $\A_{i}$ is trivial or not. If $\A_{i}$ is non-trivial, then $F_{i} \ne A_{i}$ as $\langle \A_{i}, F_{i}\rangle \in \ModS(\vdash)$. Then for every $l \in I$, we define
 \[
G_{l}\coloneqq \left\{ \begin{array}{ll}
A_{l}& \text{if $i < l$}\\
   F_{l} & \text{otherwise.}\\
  \end{array} \right.  
\]
By Lemma \ref{lem:basic-facts} we know that $G \coloneqq \bigcup_{i \in I}G_{i}$ is a $\vdash^{l}$-filter on $\PL(\A_{i})_{i \in I}$. Then choose an element $c \in A_{i} \smallsetminus F_{i}$. We have that
\begin{align*}
c\cdot^{\PL} a = c\cdot^{\A_i} a = c \in A_{i} \smallsetminus F_{i} = A_{i} \smallsetminus G_{i}
\end{align*}
and $c\cdot^{\PL} b \in A_{j} = G_{j} $. Therefore, $ c\cdot^{\PL} a \notin G \text{ and } c\cdot^{\PL} b \in G$. 
Hence we conclude that $\langle a, b \rangle\not\in\theta$, as desired. %does not belong to the Suszko congruence of $\PL(X)$ as desired.

Then we consider the case where $\A_{i}$ is trivial. We have cases: either $F_{i} = \emptyset$ or $F_{i} = A_{i}$. First suppose that $F_{i} = \emptyset$. Iterating the argument in the previous paragraph (taking $c\coloneqq a$) we obtain that $\langle a, b \rangle\notin\theta$. Then consider the case where $F_{i} = A_{i}$. Observe that in this case $\langle \A_{i}, F_{i}\rangle$ is a trivial matrix. Therefore we can apply the assumption, obtaining an element $k \in I$ such that $\A_{k}$ is non-trivial, $i < k$ and $j \nleq k$. Then for every $l \in I$ we define
 \[
G_{l}\coloneqq \left\{ \begin{array}{ll}
A_{l}& \text{if $k \lor j \leq l$}\\
   F_{l} & \text{otherwise.}\\
  \end{array} \right.  
\]
By Lemma \ref{lem:basic-facts} we know that $G \coloneqq \bigcup_{i \in I}G_{i}$ is a $\vdash^{l}$-filter on $\PL(\A_{i})_{i \in I}$. Since $\A_{k}$ is non-trivial and $\langle \A_{k}, F_{k}\rangle \in \ModS(\vdash)$, there is $c \in A_{k} \smallsetminus F_{k}$. Since $k < k\vee j$, we have that
\begin{align*}
c\cdot^{\PL} a &= c\cdot^{\A_{k}} f_{ik}(a) = c\in A_{k} \smallsetminus F_{k} = A_{k} \smallsetminus G_{k}\\
c\cdot^{\PL} b & \in A_{j \lor k} = G_{j \lor k}.
\end{align*}
Hence we conclude that $c\cdot^{\PL} a \notin G \text{ and }c\cdot^{\PL} b \in G$. But this means that $\langle a, b \rangle\notin\theta$.
 \end{proof}

Theorem \ref{th: caratterizzazione Plonka suszko ridotta} identifies the Suszko reduced models of $\vdash^{l}$ that can be expressed in terms of P\l onka sums of Suszko reduced models of $\vdash$. It is natural to wonder whether it is true that \textit{all} Suszko reduced models of $\vdash^{l}$ are of this kind. Example \ref{counterexample suszko red.} shows that this does not hold in general. A full characterization of the class of Suszko reduced models can be given for the (left) variable inclusion companions of logics which have stronger properties, such as equivalential and finitary, or having inconsistency terms. These descriptions are addressed in the following subsections.

 \subsection{Equivalential logics}\label{subse: equivalential logics}
 
 It turns out that, in the setting of finitary equivalential logics $\vdash$, the class of matrices $\ModS(\vdash^{l})$ has a very transparent description in terms of P\l onka sums, as we proceed to prove (see Theorem \ref{Thm: caratterizzazione suszko regolare}).
 
 \begin{lemma}\label{lm: ridotti di regolare sono plonka di ridotti non regolari}
Let $\vdash$ be an equivalential finitary logic with a partition function. Then
\[
\Modstar(\vdash^{l}) \subseteq \PL(\Modstar(\vdash)).
\]
%
% Let  $\vdash$ be an equivalential finitary logic, and $\langle\A, F\rangle\in\Modstar(\vdash^{r})$. There exists a  of matrices $X\subseteq\Modstar(\vdash)$ such that $\PL(X)=\langle\A, F\rangle$.
 \end{lemma}
 \begin{proof}
 Recall from Lemma \ref{lemma: p.f. della regolarizzazione} that also $\vdash^{l}$ has a partition function. Then consider $\langle\A, F\rangle\in\Modstar(\vdash^{l})$ and let
 \[
 X=\langle \langle I,\leq\rangle , \{ \langle \A_{i},F_{i}\rangle\}_{i\in I}, \{ f_{ij} \! : \! i \leq j \}\rangle  
\]
 be the directed system of matrices given in Theorem \ref{th: plonka sum of matrices}. We know that $\PL(X)=\langle\A, F\rangle$. Moreover, by Lemma \ref{lm: regolarizzazione con p.f.}, we know that each fiber of $X$ is a model of $\vdash$. It only remains to prove that the fibers of $X$ are Leibniz reduced. 
 
We claim that $\bigcup_{i\in I} \Leibniz^{\A_i}F_i $ is a congruence of $\A$. To show this, let $\Delta(x,y)$ be a set of congruence formulas for $\vdash$. Then consider an $n$-ary basic operation $\lambda$ and elements $a_{1},\dots,a_{n},b_{1},\dots,b_{n}\in A$ such that $\langle a_j,b_j\rangle\in \bigcup_{i\in I} \Leibniz^{\A_i}F_i $, for all $1\leq j\leq n$. This implies that are indexes $m_1,\dots, m_n\in I$ such that $a_j,b_j\in A_{m_j}$, for all $j\leq n$, and moreover that $\langle a_j,b_j\rangle\in \Leibniz^{\A_{m_j}}F_{m_j} $. The fact that $\Delta$ is a set of congruence formulas for $\vdash$ implies that 
\[
\Delta^{\PL}(a_j , b_j)=\Delta^{\A_{m_j}}(a_j , b_j)\subseteq F_{m_j}.
\]
Set $k\coloneqq m_{1}\vee\dots\vee m_{n}$. We have that
\begin{equation}\label{eq:delta_equiv1}
\bigcup_{j \leq n}\Delta^{\A_k}(f_{m_{j}k}(a_j), f_{m_{j}k}(b_j))\subseteq F_k
\end{equation}
From the fact that $\Delta$ is a set of congruence formulas for $\vdash$ it follows that (recall that $\lambda$ is an $n$-ary arbitrary operation)
\begin{equation}\label{eq:delta_equiv2}
\bigcup_{j \leq n}\Delta(x_j, y_j) \vdash \Delta(\lambda(\vec{x}), \lambda(\vec{y})).
\end{equation}

Together with \eqref{eq:delta_equiv1} and \eqref{eq:delta_equiv2}, the fact that $\pair{\A_k, F_k}$ is a model of $\vdash$ implies that 
\begin{align*}
&\text{ }\Delta^{\A_k}(\lambda^{\PL}(a_1, \dots, a_n), \lambda^{\PL}(b_1, \dots, b_n)) \\
=&\text{ }\Delta^{\A_k}(\lambda(f_{m_{1}k}(a_1), \dots, f_{m_{n}k}(a_n)), \lambda(f_{m_{1}k}(b_1), \dots, f_{m_{n}k}(b_n)))\\
 \subseteq &\text{ } F_k.
\end{align*}
Together with the fact that $\Delta$ is a set of congruence formulas for $\vdash$, this implies that
\[
\langle \lambda^{\PL}(\vec{a}), \lambda^{\PL}(\vec{b})\rangle \in \Leibniz^{\A_k}F_k \subseteq \bigcup_{i\in I} \Leibniz^{\A_i}F_i.
\]
This establishes the claim.

Since each $\Leibniz^{\A_i}F_i$ is compatible with $F_{i}$, we know that the congruence $\bigcup_{i\in I} \Leibniz^{\A_i}F_i$ is compatible with $F$. In particular, this implies that $\bigcup_{i\in I} \Leibniz^{\A_i}F_i \subseteq \Leibniz^{\A}F$. Since $\Leibniz^{\A}F$ is the identity relation, we conclude that so is each $\Leibniz^{\A_i}F_i$. Hence we obtain that $\langle \A_{i}, F_{i}\rangle \in \Modstar(\vdash)$ for every $i \in I$ and, therefore, that
\[
\langle \A, F \rangle = \PL(X) \subseteq \PL(\Modstar(\vdash)).
\]
We conclude that $\Modstar(\vdash^{l}) \subseteq \PL(\Modstar(\vdash))$, as desired.
 \end{proof}

\begin{corollary}\label{Cor:equivalential-partial-result}
If $\vdash$ is an equivalential finitary logic with a partition function, then
\[
\ModS(\vdash^{l}) \subseteq \PL(\Modstar(\vdash)) = \PL(\ModS(\vdash)).
\]
\end{corollary} 

\begin{proof}
First recall that $\ModS(\vdash) = \Modstar(\vdash)$, since $\vdash$ is equivalential. Thus it will be enough to prove that $\ModS(\vdash^{l}) \subseteq \PL(\Modstar(\vdash))$. We have that
\begin{align}
\ModS(\vdash^{l})& = \PSD\Modstar(\vdash^{l})\label{Eq:SP2}\\
& \subseteq \SSS\PPP\Modstar(\vdash^{l}) \label{Eq:SP3}\\
&\subseteq \SSS\PPP\PL(\Modstar(\vdash))\label{Eq:SP4}\\
& \subseteq \PL(\SSS\PPP\Modstar(\vdash))\label{Eq:SP5}\\
& = \PL(\Modstar(\vdash))\label{Eq:SP6}.
\end{align}
The non-trivial inclusions above are justified as follows: (\ref{Eq:SP4}) is a consequence of Lemma \ref{lm: ridotti di regolare sono plonka di ridotti non regolari}, (\ref{Eq:SP5}) follows from Lemma \ref{Lem:plonka-commutes-with-S-and-P}, and (\ref{Eq:SP6}) from the fact that $\Modstar(\vdash)$ is closed under $\SSS$ and $\PPP$, since $\vdash$ is equivalential. Hence we conclude that $\ModS(\vdash^{l}) \subseteq \PL(\Modstar)$. 
\end{proof}
 
We are now ready to provide a full characterization of the Suszko reduced models of the variable inclusion companion of a finitary equivalential logic (with partition function.

  \begin{theorem}\label{Thm: caratterizzazione suszko regolare}
Let $\vdash$ be an equivalential and finitary logic with a partition function, and $\langle\A, F \rangle$ be a matrix. The following conditions are equivalent:
\benroman
\item $\langle \A, F \rangle \in \ModS(\vdash^{l})$.
\item There exists a directed system of matrices $X\subseteq\Modstar(\vdash)$ indexed by a semilattice $I$ such that $\langle \A, F \rangle = \PL(X)$ and for every $n, i \in I$ such that $\langle \A_{n}, F_{n}\rangle$ is trivial and $n<i$, there exists $j \in I$ s.t. $n\leq j, i\nleq j$ and $\A_{j}$ is non-trivial.
\eroman
 \end{theorem}

 \begin{proof}
 This is a consequence of Theorem \ref{th: caratterizzazione Plonka suszko ridotta} and Corollary \ref{Cor:equivalential-partial-result}.
 \end{proof}

\begin{example}\label{Exa:equivalential-logics-with-partition}
Observe that all substructural logics \cite{GaJiKoOn07,Pao02} are finitary, equivalential, and have a partition function. The same holds for all local and global consequences of normal modal logics \cite{MK07c}. As a consequence, the above result provides a description of the Suszko reduced models of the left variable inclusion companions of all substructural and modal logics (when the latter are understood as local and global consequences of normal modal logics \cite{BlRiVe01,ChZa97,Kr99}).
\qed
\end{example}

\subsection{Inconsistency terms}\label{sec: inconsistency}

The following definition originates in \cite{ThomasLavtesi}, but see also \cite{CampercholiRaftery,JGR13}:
\begin{definition}\label{def: iconsistency terms}
A logic $\vdash$ has a set of \emph{inconsistency terms} if there is a set of formulas $\Sigma$ such that $\sigma[\Sigma] \vdash \varphi$ for every substitution $\sigma$ and formula $\varphi$.
\end{definition}

\begin{example}\label{Exa:inconsistency-terms}
For any formula $\varphi$, the set $\{ \lnot(\varphi \to \varphi) \}$ is a set of inconsistency terms for all superintuitionistic logics, all axiomatic extensions of MTL-logic \cite{CiHaNo11ab,EsGo01} including \L ukasiewicz logic \cite{CiMuOt99}, and all local and global consequences of normal modal logics.
\qed
\end{example}

\begin{remark}
Observe that if $\vdash$ has a set of inconsistency terms, then $\vdash$ has a set of inconsistency terms only in variable $x$. If, moreover, $\vdash$ is finitary, then it has a \emph{finite} set of inconsistency terms only in variable $x$. 
\qed
\end{remark}

The goal of this section is to show that if $\vdash$ is a logic with a set of inconsistency terms, then the description of the Suszko reduced models of its variable inclusion companion can be substantially improved (see Theorems \ref{th: caratterizzazione Plonka suszko ridotta} and \ref{Thm: caratterizzazione suszko regolare}), as we show in this section.

The next result discloses the semantic meaning of inconsistency terms. It should be observed that algebraic versions of it first appeared in \cite{Kollar} and \cite{CampercholiVaggione} in the setting of varieties and quasi-varieties of algebras respectively.

\begin{lemma}\label{lemma: caratterizzazione equivalente degli incons. terms}
Let $\vdash$ be a logic. The following are equivalent:
\benroman
\item $\vdash$ has a set of inconsistency terms $\Sigma$.
\item If $\langle\A,F\rangle\in\Mod(\vdash)$ is non-trivial, then it has no trivial submatrix.
\eroman
\end{lemma}
\begin{proof}
(i)$\Rightarrow$(ii): Suppose that $\vdash$ has a set of inconsistency terms $\Sigma$. We can assume w.l.o.g.\ that $\Sigma$ is in variable $x$ only. Suppose, in view of a contradiction, that there is a non-trivial matrix $\pair{\A, F}\in\Mod(\vdash)$ with a trivial submatrix $\pair{\B , B}$. Since $\langle\A,F\rangle$ is non trivial, there exists an element $a\in A\smallsetminus F$. Consider any homomorphism $v\colon\Fm\to\A$ such that $v(x)= b$ and $v(y)=a$, where $b$ is any element of $B$. Since $\Sigma=\Sigma(x)$ and $\pair{\B, B}$ is a submatrix of $\pair{\A, F}$, we have that $v[\Sigma] \subseteq B \subseteq F$. Together with the fact that $\Sigma\vdash y$, this implies that $a=v(y)\in F$, which is a contradiction.

(ii)$\Rightarrow$(i): Let $Fm(x)$ be the set of formulas in variable $x$ only. We show that $Fm(x)$ is a set of inconsistency terms for $\vdash$. To this end, consider a substitution $\sigma$ and a formula $\psi$. It is enough to show that $\sigma[Fm(x)]\vdash \psi$. Let $\varphi\coloneqq\sigma (x)$. Observe that $\sigma[Fm(x)]$ coincides with the universe of the subalgebra $\textup{Sg}^{\Fm}(\varphi)$ of $\Fm$ generated by $\varphi$. Consider the matrices 
\begin{align*}
\mathsf{M}_1\coloneqq & \pair{\Fm, \mathrm{Cn}_{\vdash}(\textup{Sg}^{\Fm}(\varphi))} \\
\mathsf{M}_2\coloneqq & \pair{\textup{Sg}^{\Fm}(\varphi), \textup{Sg}^{\Fm}(\varphi)}. 
\end{align*}
Clearly, $\mathsf{M}_1$ is a model of $\vdash$ and $\mathsf{M}_2$ a trivial submatrix of $\mathsf{M}_1$. By the assumption, we get that $\mathsf{M}_1$ is a trivial matrix, i.e. $ Fm= \mathrm{Cn}_{\vdash}(\textup{Sg}^{\Fm}(\varphi)) $. Hence we conclude that 
$$ \psi\in Fm= \mathrm{Cn}(\textup{Sg}^{\Fm}(\varphi))= \mathrm{Cn}_{\vdash}(\sigma[Fm(x)]).$$
Clearly this implies that $\sigma[Fm(x)] \vdash \psi$, as desired.
\end{proof}

Remarkably, Theorem \ref{th: caratterizzazione Plonka suszko ridotta} can be substantially improved for logics possessing a set of inconsistency terms (whose presence is essential, as shown in Example \ref{Exa:counterexample-inconsistency}): 
 
\begin{theorem}\label{th: Plonka suszo ridotta nel caso con inconsistency term}
Let $\vdash$ be a logic with a partition function and a set of inconsistency terms. For every directed system $X $ of matrices in $\ModS(\vdash)$, the following conditions are equivalent: 
\benroman
\item $\PL(X)\in\ModS(\vdash^{l})$.
\item $X$ contains at most one trivial component. 
\eroman
\end{theorem}
\begin{proof}
For the sake of simplicity, throughout the proof we set
 \[
 X=\langle \langle I,\leq\rangle , \{ \langle \A_{i},F_{i}\rangle\}_{i\in I}, \{ f_{ij} \! : \! i \leq j \}\rangle.
\]
First we claim that if a component $\pair{\A_n , F_n}$ of $X$ is trivial, then so is $\pair{\A_k, F_k}$, for every $k\geq n$. To prove this, consider a trivial component $\pair{\A_{n} , F_{n}}$ of $X$ and $k \geq n$. Observe that
\[
f_{nk}[A_n]= f_{nk}[F_n] \subseteq F_k.
\]
 Then $\pair{f_{nk}[A_n],f_{nk}[F_n]}$ is a trivial submatrix of $\pair{\A_k , F_k}$. Since $\vdash$ has a set of inconsistency terms, we can apply Lemma \ref{lemma: caratterizzazione equivalente degli incons. terms} obtaining that $\pair{\A_k , F_k}$ is trivial. This establishes the claim.

(i)$\Rightarrow$(ii): Suppose, in view of a contradiction,  that $\PL(X)\in\ModS(\vdash^{l})$ and that $X$ contains two distinct trivial components $\pair{\mathbf{1}_n , \{ 1_{n}\}}$ and $\pair{\mathbf{1}_k, \{1_{k}\}}$ (their algebraic reducts are trivial, as the components of $X$ belong to $\ModS(\vdash)$). Set $\pair{\A, F}\coloneqq\PL(X)$. Observe that, for every formula $\varphi(x,\vec{z}\?\?)$ in which $x$ really occurs, and every tuple $\vec{c}\in A$, we have that
\[
\varphi^{\A}(1_{n}, \vec{c}\?\?), \varphi^{\A}(1_{k}, \vec{c}\?\?)\in F.
\]
To prove this, observe that the element $\varphi^{\A}(1_{n}, \vec{c}\?\?)$ belongs to a component $\pair{\A_l, F_l}$ of $X$ with $n\leq l$. By the previous claim, we know that $\pair{\A_l, F_l}$ is trivial and, therefore, that $\varphi^{\A}(1_{n}, \vec{c}\?\?) \in F_l \subseteq F$, as desired. A similar argument shows that $\varphi^{\A}(1_{k}, \vec{c}\?\?) \in F$ as well. Hence for every unary polynomial function $p$ of $\A$ we have that
\[
\textup{Fg}_{\vdash^{l}}^{\A}(F \cup \{ p(1_{n})\}) = \textup{Fg}_{\vdash^{l}}^{\A}(F \cup \{ p(1_{k})\}).
\]
By Lemma \ref{lem:polynomial charact. Suszko cong.} this implies that $\langle 1_{n}, 1_{k}\rangle \in \Tarski^{\A}_{\vdash}F$. Since $\pair{\A, F} \in \ModS(\vdash^{l})$, this implies that $1_n = 1_k$, which is a contradiction.

(ii)$\Rightarrow$(i): Suppose that $X$ contains at most one trivial matrix. If $X$ contains no trivial component, then, by Theorem \ref{th: caratterizzazione Plonka suszko ridotta}, we obtain that $\PL(X)\in\ModS(\vdash^{l})$. Then consider the case where $X$ contains exactly one trivial component. By the claim we obtain that this component is the maximum of $\langle I, \leq \rangle$. Again, with an application of Theorem \ref{th: caratterizzazione Plonka suszko ridotta}, we conclude that $\PL(X)\in\ModS(\vdash^{l})$.
\end{proof}
The assumption on the existence of a set of inconsistency terms for the logic $\vdash$ in the above theorem is essential, as shown in Example \ref{Exa:counterexample-inconsistency}. 

Drawing consequences from Theorem \ref{th: Plonka suszo ridotta nel caso con inconsistency term}, we obtain a very transparent description of the Suszko reduced models of the variable inclusion companion of a finitary equivalential logic with a partition function and inconsistency terms:

\begin{theorem}\label{Thm: caratterizzazione suszko regolare2}
Let $\vdash$ be an equivalential and finitary logic with a partition function and inconsistency terms, and $\langle\A, F \rangle$ be a matrix. The following conditions are equivalent:
\benroman
\item $\langle \A, F \rangle \in \ModS(\vdash^{l})$.
\item There exists a directed system of matrices $X\subseteq\Modstar(\vdash)$ with at most one trivial component such that $\langle \A, F \rangle = \PL(X)$.
\eroman
 \end{theorem}

 \begin{proof}
 This is a combination of Theorems \ref{th: Plonka suszo ridotta nel caso con inconsistency term} and \ref{Thm: caratterizzazione suszko regolare}.
 \end{proof}

\begin{example}
It is worth to observe that the above result provides a full description of the Suszko reduced models of the left variable inclusion companions of most well-known logics, including all logics mentioned in Example \ref{Exa:inconsistency-terms}.
\qed
\end{example}

%It has already been observed (see \cite[Theorem 63]{Bonzio16}) for the regularization of classical logic, namely the logic $\PWK$, that the algebraic reducts, $\Alg(\vdash_{\PWK})$, of Suszko reduced models consist of P\l onka sums over direct systems of Boolean algebras with at most one trivial component. However, the description given in Theorem \ref{th: Plonka suszo ridotta nel caso con inconsistency term} is more general as it encompasses models, instead of their algebraic reducts only.

%\begin{example}\label{ex: modelli ridotti di PWK}
%Paraconsistent Weak Kleene logic is a logic of variable inclusion whose Suszko reduced models are P\l onka sums of Boolean algebras with at most one trivial element (see \cite[Theorem 63]{Bonzio16}).
%\end{example}

\section{Classification in the Leibniz hierarchy}\label{sec: Leibniz hierarchy}

%In this section, we classify logics of variables inclusion in the Leibniz hierarchy. ......

We conclude this work by investigating the location of logics of variable inclusion in the Leibniz hierarchy. To this end, recall that a logic $\vdash$ is inconsistent if $\Gamma \vdash \varphi$ for every $\Gamma \cup \{ \varphi \} \subseteq Fm$. Equivalently, $\vdash$ is inconsistent if $\emptyset \vdash x$ for some variable $x$. A logic is \textit{consistent} when it is not inconsistent.

\begin{theorem}\label{the regularisation is not protoalg}
Let $\vdash$ be a logic.
\benroman
\item If $\vdash$ is consistent, then $\vdash^{l}$ is not protoalgebraic.
\item If $\vdash$ is finitary, algebraizable and has a partition function, then $\vdash^{l}$ is truth-equational.
\eroman
\end{theorem}
\begin{proof}
(i): We reason by contraposition. Suppose that $\vdash^{l}$ is protoalgebraic. Then there is a set of formulas $\Delta(x, y)$ such that $\emptyset \vdash^{l}\Delta(x, x)$ and $x, \Delta(x, y) \vdash^{l} y$. Thus, the definition of $\vdash^{l}$ implies that there is a subset $\Sigma(y) \subseteq \Delta(x, y)$ such that $\Sigma(y) \vdash y$. Since $\emptyset \vdash^{l} \Delta(x, x)$, we have that $\emptyset \vdash^{l} \Sigma(y)$. From $\Sigma(y) \vdash y$ and $\emptyset \vdash^{l} \Sigma(y)$ it follows that $\emptyset \vdash^{l} y$. By the definition of $\vdash^{l}$ we conclude that $\emptyset \vdash y$ and, therefore, that $\vdash$ is inconsistent.

(ii): Suppose that $\vdash$ is finitary, algebraizable and has a partition function. In particular, $\vdash$ is truth-equational with set of defining equations $\boldsymbol{\tau}(x)$. We will show that $\boldsymbol{\tau}(x)$ is a set of defining equations for $\vdash^{l}$ as well. To this end, consider $\langle \A, F \rangle \in \Modstar(\vdash^{l})$. Since $\vdash$ is finitary, equivalential and with a partition function, we can apply Lemma \ref{lm: ridotti di regolare sono plonka di ridotti non regolari} obtaining that there exists a directed system of matrices $X \subseteq \Modstar(\vdash)$ such that $\langle \A, F \rangle = \PL(X)$. For the sake of simplicity, we set
 \[
 X=\langle \langle I,\leq\rangle , \{ \langle \A_{i},F_{i}\rangle\}_{i\in I}, \{ f_{ij} \! : \! i \leq j \}\rangle
\]
and assume w.l.o.g.\ that $\langle \A, F \rangle = \PL(X)$. Consider an element $a \in A$. There is $i \in I$ such that $a \in A_{i}$. We have that
\begin{equation}\label{Eq:defining-equations-in-plonka}
\A \vDash \boldsymbol{\tau}(a) \Longleftrightarrow \A_{i} \vDash \boldsymbol{\tau}(a) \Longleftrightarrow a \in F_{i} \Longleftrightarrow a \in F.
\end{equation}
The above equivalences are justified as follows. The first one follows from the fact that $\A = \PL(\A_{i})_{i \in I}$. The second one follows from the fact that $\langle \A_{i}, F_{i}\rangle \in \Modstar(\vdash)$ and that $\boldsymbol{\tau}(x)$ is a set of defining equations for $\vdash$. The last one follows from the observation that $\langle \A, F \rangle = \PL(X)$.

By (\ref{Eq:defining-equations-in-plonka}) we obtain that for every $a \in A$,
\[
\A \vDash \boldsymbol{\tau}(a) \Longleftrightarrow a \in F.
\]
Hence we conclude that $\boldsymbol{\tau}(x)$ is a set of defining equations for $\vdash^{l}$ and, therefore, $\vdash^{l}$ is truth-equational.
\end{proof}

In \cite[Theorem 48]{Bonzio16} it is proved that the variety of involutive bisemilattices, i.e.\ the closure under Plonka sums of the variety of Boolean algebras \cite{tesiLuisa}, is not the equivalent algebraic semantics of \emph{any} algebraizable logic. This result can be strengthened as follows:

\begin{theorem}
Let $\mathsf{K}$ be a class of algebras containing two trivial algebras and closed under P\l onka sums. There is no protoalgebraic logic $\vdash$ such that $\Alg(\vdash) = \mathsf{K}$.
\end{theorem}
\begin{proof}
Suppose, in view of a contradiction, that there are a class of algebras $\mathsf{K}$ containing two trivial algebras and closed under P\l onka sums, and a protoalgebraic logic $\vdash$ such that $\Alg(\vdash) = \mathsf{K}$. 
Let $\mathbf{1}_a, \mathbf{1}_b \in \mathsf{K} $ be distinct trivial algebras and consider the directed system obtained by the homomorphism $f_{a b} \colon \mathbf{1}_{a} \to \mathbf{1}_{b}$ ($a\leq b$ in the semilattice order of the indexes). Let $\A=\mathbf{1}_{a} \oplus \mathbf{1}_{b}$ be the P\l onka sum of this directed system.
% Observe that $A = \mathbf{1}_a + \mathbf{1}_b \in \PL(\mathsf{K}) = \mathsf{K} = \Alg(\vdash) $. As $\Alg(\vdash)$ contains a non-trivial algebra, it is not hard to see that $ x \nvdash y$.} \deleted[id=Stef]{Since $\Alg(\vdash) = \mathsf{K}$, the same holds for $\Alg(\vdash)$. It is not difficult to see that this implies $x \nvdash y$}.
Clearly $\A \in \mathsf{K}$. Therefore there is $F \subseteq A$ such that $\langle \A, F \rangle \in \ModS(\vdash)$. % since $\mathsf{K}$ contains two trivial algebras and is closed under P\l onka sums. 
As $\Alg(\vdash)$ contains a non-trivial algebra, it is not difficult to see that $x \nvdash y$. Since $\vdash$ is protoalgebraic, there is a set of formulas $\Delta(x, y)$ such that $\emptyset \vdash \Delta(x, x)$ and $x, \Delta(x, y) \vdash y$. Since $x \nvdash y$ and $x, \Delta(x, y) \vdash y$, we conclude that $\Delta(x, y) \ne \emptyset$. Then consider $\varphi(x, y) \in \Delta(x, y)$. Since $\emptyset \vdash \Delta(x, x)$, we conclude that $\emptyset \vdash \varphi(x, x)$. 

Now, observe that the variable $x$ really occurs in $\varphi(x, x)$, since we do not allow the presence of constant symbols in this paper. Hence we obtain that
\[
\varphi^{\A}(1_{a}, 1_{a}) = 1_{a} \text{ and }\varphi^{\A}(1_{b}, 1_{b}) = 1_{b}.
\]
Together with the fact that $\emptyset \vdash \varphi(x, x)$, this implies that $A=\{1_{a},1_{b}\}$ is the smallest $\vdash$-filter on $\A$. In particular, this implies that  $A$ is the \textit{unique} $\vdash$-filter on $\A$. Since $F$ is a $\vdash$-filter on $\A$, we conclude that $A = F$. Hence $\langle \A, A\rangle$ is a Suszko reduced model of $\vdash$. This implies that $\A$ is trivial, which is false.
\end{proof}

\section*{Appendix}

Aim of this section is showing that some of the assumptions are indeed essential in order to prove certain results. In particular, the following example shows that, in general, there can be Suszko reduced models of the logic $\vdash^{l}$ that are not P\l onka sums of Suszko reduced models of $\vdash$.

\begin{example}\label{counterexample suszko red.}
Consider the logic $\vdash$  determined by the following class of matrices:
\[
\mathsf{M} \coloneqq \{ \langle \A, F \rangle : \A \text{ is a distributive lattice and }F \text{ is an upset}\}.
\]
Let $\A_{1}$ be the three element lattice $a < b < c$ and let $F_{1} = \{ b, c \}$. Moreover, let $\A_{2}$ be the four-element Boolean lattice (with universe $\{ 0, d, e, 1\}$ with $0$ as bottom element), and let $F_{2} = A_{2} \smallsetminus \{ 0 \}$. Clearly both $\langle \A_{1}, F_{1}\rangle$ and $\langle \A_{2}, F_{2}\rangle$ are models of $\vdash$ (as they belong to $\mathsf{M}$). However, it is easy to see that $\langle \A_{1}, F_{1}\rangle \notin \ModS(\vdash)$. Now, let $f \colon \A_{1} \to \A_{2}$ be any of the two embeddings of $\A_{1}$ into $\A_{2}$. Clearly these two matrices plus $f$ give rise to a directed system $X$ of matrices (of course one should pedantically add the identity endomorphisms) depicted in the following figure. We denote by $\langle \B, G\rangle$ the P\l onka sum $\PL(X)$.

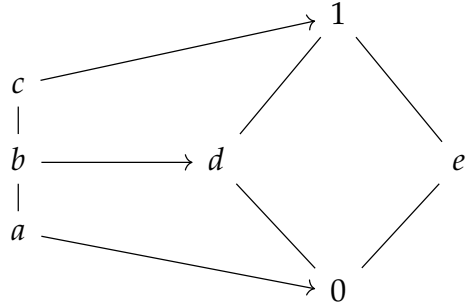
\begin{center}
\[\begin{tikzcd}[row sep = tiny]%, arrows = {dash}]
& & & 1 &\\
& & & &\\
c\arrow[uurrr] & & & & \\
& & & & \\
b\arrow[uu, dash]\arrow[rr] & & d\arrow[uuuur, dash] & & e\arrow[uuuul, dash] \\
& & & & \\
a\arrow[uu, dash]\arrow[drrr] & & & & \\
& &  & 0\arrow[uuul, dash]\arrow[uuur, dash] & \\
& & & & 
\end{tikzcd}\]
\end{center}

Since $\langle \A_{1}, F_{1}\rangle$ and $\langle \A_{2}, F_{2}\rangle$ are models of $\vdash$, by Lemma \ref{lem:soundness} $\langle \B, G\rangle$ is a model of $\vdash^{l}$. We now show that it is indeed Suszko reduced. Elements belonging to the algebra $\A_1$, as for example $b$ and $c$ (any other pair of elements in $\A_1$ is distinguished by the identity function), can be distinguished by means of the function $\wedge^{\B}$, the filter G and the element $e$, as follows: 
\[ 
b\wedge^{\B} e = d\wedge^{\A_{2}} e = 0\not\in G 
\]  
\[ 
c\wedge^{\B} e = 1\wedge^{\A_{2}} e = e\in G. 
\]  
One can reason similarly (using $G$ as filter) for pairs of elements belonging to $\A_2$ (we illustrate the only interesting case):
\[ 
d\wedge^{\B} b = d\wedge^{\A_{2}} d = d\in G 
\]  
\[ 
e\wedge^{\B} b = e\wedge^{\A_{2}} d = 0\not\in G. 
\]  

On the other hand, pairs of elements belonging to different algebras are distinguished by considering the filter $H\coloneqq F_1\cup A_2$ on $\B$ (the fact that it is a filter is guaranteed by Lemma \ref{lem:basic-facts}) , the function $\wedge^{\B}$ and the element $a$. Consider, for instance, the elements $b$ and $d$:
\[
b\wedge^{\B} a = a\not\in H; 
\] 
\[
d\wedge^{\B} a = d\wedge^{\A_2} 0 = 0\in H. 
\] 
\noindent
This is enough to show that $\pair{\B , G}$ is Suszko reduced.

To conclude the example we need to disprove that $\langle \B, G \rangle$ is a P\l onka sum of any Suszko reduced models of $\vdash$. Suppose that $\langle \B, G\rangle$ is the P\l onka sum of a directed system $Y$ of Suszko reduced models $\langle \B_{1}, G_{1}\rangle, \dots, \langle \B_{n}, G_{n}\rangle$ of $\vdash$. First observe that $n \leq 2$. Suppose the contrary towards a contradiction. Then $n \geq 3$. We choose three elements $b_{1} \in B_{1}, b_{2} \in B_{2}$ and $b_{3} \in B_{3}$. Clearly $b_{1}, b_{2}$ and $b_{3}$ are different. Moreover, for every $1 \leq i  < j \leq 3$ we have that either $b_{i}\cdot^{\B} b_{j} \ne b_{i}$ or $b_{j}\cdot^{\B} b_{i} \ne b_{j}$, where $\cdot$ indicates the partition function, i.e. $x\cdot y\coloneqq x\wedge(x\lor y)$. It is easy to see that no such three elements exist in $\B$, which is a contradiction. Hence $n \leq 2$. We have cases. If $n =1$, then $\langle \B_{1}, G_{1}\rangle = \langle \B, G\rangle$. In particular, this implies that $\langle \B, G\rangle \in \ModS(\vdash)$ and, therefore, $\B \in \Alg(\vdash)$. By Lemma \ref{lem:algebraicreducts} this implies that $\B$ is a lattice, which is false. Thus, the only possible case is that $n = 2$. Now, by Lemma \ref{lem:algebraicreducts} we know that $\B_{1}$ and $\B_{2}$ are distributive lattices. Since the only way of partitioning $\B$ into two subalgebras that are distributive lattices is $\{ \A_{1}, \A_{2}\}$, we conclude that w.l.o.g.\ $\B_{1} = \A_{1}$ and $\B_{2} = \A_{2}$, i.e. $\langle \B, G \rangle$ can not be the P\l onka sum of any Suszko reduced models of $\vdash$.
\qed
\end{example}

%\added[id=Stef]{The following example shows that the presence of a set of inconsistency terms is essential in Theorem \ref{th: Plonka suszo ridotta nel caso con inconsistency term}.}

\begin{example}\label{Exa:counterexample-inconsistency}
The statement of Theorem \ref{th: Plonka suszo ridotta nel caso con inconsistency term} is in general false for logics without a set of inconsistency terms, as witnessed by the following example based on $\CL^{\land\lor}$, the conjunction and disjunction fragment of classical propositional logic (which does not possess a set of inconsistency terms). In particular, it happens to have a Suszko reduced model of $\vdash^{l}$, which is the P\l onka sum of a directed system of Suszko reduced models of $\vdash$ containing two trivial matrices.

Let $\vdash$ be the $\land,\lor$-fragment of classical propositional logic. Moreover, let $\mathbf{1}$ be the trivial lattice and $\mathbf{L}_{2}=\pair{\{ \perp,\top\}, \land, \lor}$ the 2-element distributive lattice (with $\perp < \top$). Consider the directed system $X$ of matrices formed by 6 copies of the matrix $\langle \mathbf{L}_{2}, \{\top\}\rangle$ and two trivial matrices $\langle\mathbf{1},\{1\}\rangle$ sketched in the following figure 
%$\langle\mathbf{L}_{b},\{b\}\rangle$, where $a,b$ are the unique elements of $\mathbf{L}_{a},\mathbf{L}_{b}$, respectively,  and $\top_{2}$ is the top element of $\mathbf{L}_{2}$, as drawn in the picture below 
(lines represent lattice order in the P\l onka fibers, arrows, the homomorphisms, and circles, filters in any fiber). 

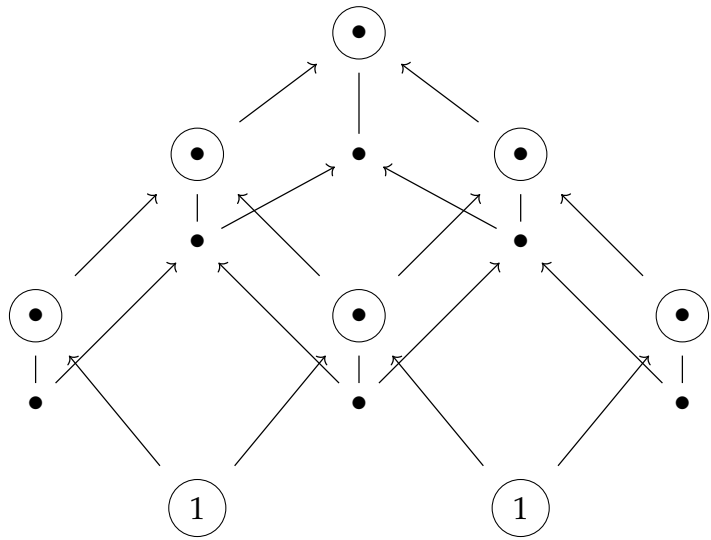
\[\begin{tikzcd}[row sep = tiny, arrows = {dash}]
& & & \circled{$\bullet$} &  & &  \\
& & & &  & & \\
& & & &  & & \\
& &  \circled{$\bullet$}\ar[uuur, ->] & \bullet\ar[uuu] & \circled{$\bullet$}\ar[uuul,->]  &  & \\
& & & &  & &  \\
& & \bullet\ar[uu]\ar[uur,->] &  & \bullet\ar[uu]\ar[uul,->] &  & \\
%& & & &  & &  \\
 & \circled{$\bullet$}\ar[uuur,->] & & \circled{$\bullet$}\ar[uuul,->]\ar[uuur,->] & & \circled{$\bullet$}\ar[uuul,->] & \\
 & & & &  & &  \\
 & \bullet\ar[uu]\ar[uuur,->] & & \bullet\ar[uu]\ar[uuur,->]\ar[uuul,->] & & \bullet\ar[uu]\ar[uuul,->] & \\
 & & & &  & &  \\
 & & & &  & &  \\
 &  & \circled{$1$}\ar[uuuuur,->]\ar[uuuuul,->] &  & \circled{$1$}\ar[uuuuur,->]\ar[uuuuul,->] & & \\
\end{tikzcd}\]

Clearly each matrix in $X$, which contains two trivial matrices, is a Suszko reduced model of $\vdash$. Moreover, by applying Theorem \ref{th: caratterizzazione Plonka suszko ridotta}, one immediately checks that $\PL(X)\in\ModS(\vdash^{l})$.
% By applying the filter construction described in point (iii) of Lemma \ref{lem:basic-facts} it is easy to see that for every pair of elements $p,q\in\bigcup_{i\in I }L_{i}$ there exists a $\vdash^{r}$-filter $G$ over $\PL(X)$ extending $\bigcup_{i\in I}F_{i}$ and an element $c\in\bigcup_{i\in }L_{i}$ such that $p\land c\in G$ if and only if $q\land c\notin G$. For example, consider the elements $a,b$. Define
%\[
%G_{i}\coloneqq \left\{ \begin{array}{ll}
%A_{i}& \text{if $a<i$}\\
%   F_{i} & \text{otherwise.}\\
%  \end{array} \right.  
%\]
% By Lemma \ref{lem:basic-facts} $\bigcup_{i\in I}G_{i}$ is a $\vdash^{r}$-filter over $\PL(X)$. Consider the element  $c\in A_{j}\smallsetminus F_{j}$ with $a\nleq j$. Cleraly $a\land c\in\bigcup_{i\in I}G_{i}$ while  $b\land c\notin\bigcup_{i\in I}G_{i}$
%
% This strategy can be applied for all the elements in $\PL(X)$ by defining $G_{i}$ appropriately. That is, $\Tarski^{\PL(X)}\bigcup_{i\in I}F_{i}=\text{id}$. 
\qed
%By using a similar strategy to the one described in detail in Example \ref{counterexample equivalentiality to theorem suszko reduced}, it is not difficult to check that the model is Suszko reduced as any two elements can be distinguished by the partition function and an appropriately chosen filter. \textcolor{blue}{Forse va dettagliato...}
\end{example}

\section*{Acknowledgments}

The first and the second author were both supported by the grant GBP202/12/G061 of the Czech Science Foundation. The first author acknowledges also the ERC grant: ``Philosophy of Pharmacology: Safety, Statistical Standards, and Evidence Amalgamation'', GA:639276. The second author was supported also by a Beatriz Galindo fellowship of the Ministry of Education and Vocational Training of the Government of Spain. We are grateful to an anonymous referee for his/her valuables comments and suggestions.
 
 %\nocite{Bergman2015}
 
%\bibliographystyle{abbrv}
%\bibliography{IBSL}

\begin{thebibliography}{10}

\bibitem{Be11g}
C.~Bergman.
\newblock {\em Universal Algebra: Fundamentals and Selected Topics}.
\newblock Chapman and Hall/CRC, 2011.

\bibitem{Bergman2015}
C.~Bergman and D.~Failing.
\newblock Commutative idempotent groupoids and the constraint satisfaction
  problem.
\newblock {\em Algebra universalis}, 73(3):391--417, 2015.

\bibitem{BlRiVe01}
P.~Blackburn, M.~de~Rijke, and Y.~Venema.
\newblock {\em Modal logic}.
\newblock Cambridge University Press, 2001.

\bibitem{BP89}
W.~Blok and D.~Pigozzi.
\newblock {\em Algebraizable logics}.
\newblock American Mathematical Society, 1989.

\bibitem{BP86}
W.~J. Blok and D.~Pigozzi.
\newblock Protoalgebraic logics.
\newblock {\em Studia Logica}, 45:337--369.

\bibitem{BP92}
W.~J. Blok and D.~Pigozzi.
\newblock Algebraic semantics for universal {H}orn logic without equality.
\newblock In A.~Romanowska and J.~Smith, editors, {\em Universal Algebra and
  Quasigroup Theory}, pages 1--56. Heldermann, 1992.

\bibitem{Bochvar}
D.~Bochvar.
\newblock On a three-valued calculus and its application in the analysis of the
  paradoxes of the extended functional calculus.
\newblock {\em Mathematicheskii Sbornik}, 4:287--308, 1938.

\bibitem{BoemBonzio}
F.~Boem, S.~Bonzio, and B.~Osimani.
\newblock The logic of scientific attitude.
\newblock Submitted, 2020.

\bibitem{SB18}
S.~Bonzio.
\newblock Dualities for {P}{\l}onka sums.
\newblock {\em Logica Universalis}, 12(3):327--339, 2018.

\bibitem{Bonzio16}
S.~Bonzio, J.~Gil-F\'erez, F.~Paoli, and L.~Peruzzi.
\newblock On {P}araconsistent {W}eak {K}leene {L}ogic: axiomatization and
  algebraic analysis.
\newblock {\em Studia Logica}, 105(2):253--297, 2017.

\bibitem{Loi}
S.~Bonzio, A.~Loi, and L.~Peruzzi.
\newblock A duality for involutive bisemilattices.
\newblock {\em Studia Logica}, 107(2):423--444, 2019.

\bibitem{BonzioPraBa}
S.~Bonzio and M.~Pra~Baldi.
\newblock Logic of right variables inclusion and {P}\l onka sums of matrices.
\newblock Submitted, 2019.

\bibitem{BonzioValota}
S.~Bonzio, M.~Pra~Baldi, and D.~Valota.
\newblock Counting finite linearly ordered involutive bisemilattices.
\newblock In J.~Desharnais, W.~Guttmann, and S.~Joosten, editors, {\em
  Relational and Algebraic Methods in Computer Science}, pages 166--183.
  Springer, 2018.

\bibitem{BuSa00}
S.~Burris and H.~P. Sankappanavar.
\newblock {\em A course in {U}niversal {A}lgebra}.
\newblock The millennium edition, 2012.

\bibitem{CampercholiRaftery}
M.~A. Campercholi and J.~G. Raftery.
\newblock Relative congruence formulas and decompositions in quasivarieties.
\newblock {\em Algebra universalis}, 78(3):407--425, 2017.

\bibitem{CampercholiVaggione}
M.~A. Campercholi and D.~J. Vaggione.
\newblock Implicit definition of the quaternary discriminator.
\newblock {\em Algebra universalis}, 68(1):1--16, 2012.

\bibitem{ChZa97}
A.~Chagrov and M.~Zakharyaschev.
\newblock {\em Modal Logic}.
\newblock Oxford University Press, 1997.

\bibitem{CiMuOt99}
R.~Cignoli, I.~M.~L. D'Ottaviano, and D.~Mundici.
\newblock {\em Algebraic foundations of many-valued reasoning}.
\newblock Kluwer Academic Publishers, 2000.

\bibitem{CiHaNo11ab}
P.~Cintula, P.~H{\'a}jek, and C.~Noguera, editors.
\newblock {\em Handbook of Mathematical Fuzzy Logic. Volumes 1 and 2}.
\newblock Studies in Logic. Mathematical Logic and Foundations. College
  Publications, 2011.

\bibitem{CiuniCarrara}
R.~Ciuni and M.~Carrara.
\newblock Characterizing logical consequence in paraconsistent weak {K}leene.
\newblock In L.~Felline, A.~Ledda, F.~Paoli, and E.~Rossanese, editors, {\em
  New Directions in Logic and the Philosophy of Science}, pages 165--176.
  College Publications, 2016.

\bibitem{Ciuni2}
R.~Ciuni, T.~M. Ferguson, and D.~Szmuc.
\newblock Logics based on linear orders of contaminating values.
\newblock {\em Journal of Logic and Computation}, 29(5):631--663, 2019.

\bibitem{Cz01}
J.~Czelakowski.
\newblock {\em Protoalgebraic logics}.
\newblock Kluwer Academic Publishers, 2001.

\bibitem{Szmuctruth}
B.~Da~{R}\'e, F.~Pailos, and D.~Szmuc.
\newblock {Theories of truth based on four-valued infectious logics}.
\newblock {\em Logic Journal of the IGPL}, 2018.

\bibitem{DeWi02}
K.~Denecke and S.~L. Wismath.
\newblock {\em Universal algebra and applications in theoretical computer
  science}.
\newblock Chapman and amp, 2002.

\bibitem{Di65}
A.~Diego.
\newblock {\em Sobre {\'a}lgebras de {H}ilbert}, volume~12 of {\em Notas de
  L{\'o}gica Matem{\'a}tica}.
\newblock Universidad Nacional del Sur, Bah{\'\i}a Blanca (Argentina), 1965.

\bibitem{EsGo01}
F.~Esteva and L.~Godo.
\newblock Monoidal t-norm based logic: towards a logic for left-continuous
  t-norms.
\newblock 124:271--288, 2001.

\bibitem{Ferguson}
T.~Ferguson.
\newblock A computational interpretation of conceptivism.
\newblock {\em Journal of Applied Non-Classical Logics}, 24(4):333--367, 2014.

\bibitem{Font16}
J.~M. Font.
\newblock {\em Abstract Algebraic Logic: An Introductory Textbook}.
\newblock College Publications, 2016.

\bibitem{FJa09}
J.~M. Font and R.~Jansana.
\newblock {\em A general algebraic semantics for sentential logics}.
\newblock A.S.L., 2009.

\bibitem{FJaP03b}
J.~M. Font, R.~Jansana, and D.~Pigozzi.
\newblock A survey on abstract algebraic logic.
\newblock {\em Studia Logica, Special Issue on Abstract Algebraic Logic, Part
  {II}}, 74(1--2):13--97, 2003.

\bibitem{GaJiKoOn07}
N.~Galatos, P.~Jipsen, T.~Kowalski, and H.~Ono.
\newblock {\em Residuated Lattices: an algebraic glimpse at substructural
  logics}.
\newblock Elsevier, Amsterdam, 2007.

\bibitem{GR91}
G.~Gierz and A.~Romanowska.
\newblock Duality for distributive bisemilattices.
\newblock {\em Journal of the Australian Mathematical Society, A}, 51:247--275,
  1991.

\bibitem{Hallden}
S.~Halld\'en.
\newblock {\em The Logic of Nonsense}.
\newblock Uppsala Universitets \AA rsskrift, 1949.

\bibitem{Harding2016}
J.~Harding and A.~Romanowska.
\newblock Varieties of {B}irkhoff systems: {P}art {I}.
\newblock {\em Order}, 34(1):45--68, 2017.

\bibitem{Harding20172}
J.~Harding and A.~Romanowska.
\newblock Varieties of {B}irkhoff systems: Part {II}.
\newblock {\em Order}, 34(1):69--89, 2017.

\bibitem{JanMorI}
R.~Jansana and T.~Moraschini.
\newblock The poset of all logics {I}: {I}nterpretations and lattice structure.
\newblock Submitted, 2019.

\bibitem{JanMorII}
R.~Jansana and T.~Moraschini.
\newblock The poset of all logics {II}: {L}eibniz classes and hierarchy.
\newblock Submitted, 2019.

\bibitem{JanMorIII}
R.~Jansana and T.~Moraschini.
\newblock The poset of all logics {III}: {F}initely presentable logics.
\newblock {\em Studia Logica}, 2020.
\newblock To appear.

\bibitem{Kal71}
J.~Kalman.
\newblock Subdirect decomposition of distributive quasilattices.
\newblock {\em Fundamenta Mathematicae}, 2(71):161--163, 1971.

\bibitem{Kleene}
S.~Kleene.
\newblock {\em Introduction to Metamathematics}.
\newblock North Holland, Amsterdam, 1952.

\bibitem{Kollar}
J.~Koll\'ar.
\newblock Congruences and one-element subalgebras.
\newblock {\em Algebra Universalis}, 9:266--267, 1979.

\bibitem{Kr99}
M.~Kracht.
\newblock {\em Tools and techniques in modal logic}.
\newblock North-Holland Publishing Co., Amsterdam, 1999.

\bibitem{MK07c}
M.~Kracht.
\newblock {\em Modal consequence relations}, chapter 8 of the {H}andbook of
  {M}odal {L}ogic.
\newblock Elsevier Science Inc., 2006.

\bibitem{Lakser72}
H.~Lakser, R.~Padmanabhan, and C.~R. Platt.
\newblock Subdirect decomposition of {P}\l onka sums.
\newblock {\em Duke Math. J.}, 39:485--488, 1972.

\bibitem{ThomasLavtesi}
T.~L\'{a}vi\v{c}ka.
\newblock {\em An {A}bstract {S}tudy of {C}ompleteness in {I}nfinitary
  {L}ogics}.
\newblock PhD Thesis, Charles University, 2018.

\bibitem{Ledda2018}
A.~Ledda.
\newblock Stone-type representations and dualities for varieties of
  bisemilattices.
\newblock {\em Studia Logica}, 106(2):417--448, 2018.

\bibitem{Libkin}
L.~Libkin.
\newblock {\em Aspects of Partial Information in Databases}.
\newblock PhD Thesis, University of Pennsylvania, 1994.

\bibitem{McMcTa87}
R.~N. McKenzie, G.~F. McNulty, and W.~F. Taylor.
\newblock {\em Algebras, lattices, varieties. {V}ol. {I}}.
\newblock Wadsworth \& Brooks/Cole Advanced Books \& Software, 1987.

\bibitem{Pao02}
F.~Paoli.
\newblock {\em Substructural logics: a primer}, volume~13 of {\em Trends in
  Logic---Studia Logica Library}.
\newblock Kluwer Academic Publishers, Dordrecht, 2002.

\bibitem{tesiLuisa}
L.~Peruzzi.
\newblock {\em Algebraic approach to paraconsistent weak Kleene logic}.
\newblock PhD Thesis, University of Cagliari, 2018.

\bibitem{Plo67}
J.~P\l{}onka.
\newblock On a method of construction of abstract algebras.
\newblock {\em Fundamenta Mathematicae}, 61(2):183--189, 1967.

\bibitem{Plo67a}
J.~P\l{}onka.
\newblock On distributive quasilattices.
\newblock {\em Fundamenta Mathematicae}, 60:191--200, 1967.

\bibitem{Romanowska92}
J.~P{\l}onka and A.~Romanowska.
\newblock Semilattice sums.
\newblock In A.~Romanowska and J.~Smith, editors, {\em Universal Algebra and
  Quasigroup Theory}, pages 123--158. Heldermann, 1992.

\bibitem{Priestfirst}
G.~Priest.
\newblock The logic of paradox.
\newblock {\em Journal of Philosophical Logic}, 8:219--241, 1979.

\bibitem{Prior}
A.~Prior.
\newblock {\em Time and {M}odality}.
\newblock Oxford University Press, 1957.

\bibitem{Puhlmann}
H.~Puhlmann.
\newblock The snack powerdomain for database semantics.
\newblock In A.~M. Borzyszkowski and S.~Soko{\l}owski, editors, {\em
  Mathematical Foundations of Computer Science 1993}, pages 650--659. Springer,
  1993.

\bibitem{JGRa11}
J.~G. Raftery.
\newblock A perspective on the algebra of logic.
\newblock {\em Quaestiones Mathematicae}, 34:275--325, 2011.

\bibitem{JGR13}
J.~G. Raftery.
\newblock Inconsistency lemmas in algebraic logic.
\newblock {\em Mathematical Logic Quarterly}, 59(6):393--406, 2013.

\bibitem{romanowska2002modes}
A.~Romanowska and J.~Smith.
\newblock {\em Modes}.
\newblock World Scientific, 2002.

\bibitem{Romanowska97}
A.~Romanowska and J.~D. Smith.
\newblock Duality for semilattice representations.
\newblock {\em Journal of Pure and Applied Algebra}, 115(3):289--308, 1997.

\bibitem{Szmuc}
D.~Szmuc.
\newblock Defining {LFI}s and {LFU}s in extensions of infectious logics.
\newblock {\em Journal of {A}pplied non {C}lassical {L}ogics}, 26(4):286--314,
  2016.

\bibitem{Urquhart2001}
A.~Urquhart.
\newblock Basic many-valued logic.
\newblock In D.~M. Gabbay and F.~Guenthner, editors, {\em Handbook of
  Philosophical Logic - volume 2}, pages 249--295. Springer, 2001.

\bibitem{W88}
R.~W{\'o}jcicki.
\newblock {\em Theory of logical calculi. {B}asic theory of consequence
  operations}.
\newblock Reidel, Dordrecht, 1988.

\end{thebibliography}

\end{document}